\newif\ifsiam

\siamfalse

\newif\ifarxiv
\ifsiam \arxivfalse \else \arxivtrue \fi

\ifsiam
\documentclass[review,onefignum,onetabnum]{siamonline220329}
\fi

\ifarxiv
\documentclass[english]{myarticle}
\fi

\usepackage{lipsum}
\usepackage{amsfonts}
\usepackage{graphicx}
\usepackage{epstopdf}
\usepackage{algorithmic}
\usepackage{mathtools,thmtools}
\usepackage{amsmath}
\usepackage{float}
\usepackage{amsopn}
\usepackage{physics}
\usepackage{mathtools,thmtools}
\ifpdf
\DeclareGraphicsExtensions{.eps,.pdf,.png,.jpg}
\else
\DeclareGraphicsExtensions{.eps}
\fi

\usepackage{enumitem}
\setlist[enumerate]{leftmargin=.5in}
\setlist[itemize]{leftmargin=.5in}

\ifsiam
\newsiamremark{remark}{Remark}
\newsiamremark{hypothesis}{Hypothesis}
\newsiamthm{claim}{Claim}
\fi

\ifarxiv

\fi

\crefname{hypothesis}{Hypothesis}{Hypotheses}

\ifsiam
\headers{}{}
\fi

\title{Periodic Center Manifolds for Nonhyperbolic Limit Cycles in ODEs}

\author{Bram Lentjes\thanks{Department of Mathematics, Hasselt University, Diepenbeek Campus, 3590 Diepenbeek, Belgium \email{(bram.lentjes@uhasselt.be)}.}
	\and Mattias Windmolders\thanks{Department of Mathematics, KU Leuven, 3000 Leuven, Belgium  \email{(mattias.windmolders@student.kuleuven.be)}.}
	\and Yuri A. Kuznetsov\thanks{Department of Mathematics, Utrecht University, 3508 TA Utrecht, The Netherlands and Department of Applied
		Mathematics, University of Twente, 7500 AE Enschede, The Netherlands \email{(i.a.kouznetsov@uu.nl)}.}
}

\DeclareMathOperator{\spn}{span}
\DeclareMathOperator{\ran}{ran}
\DeclareMathOperator{\loc}{loc}
\DeclareMathOperator{\BC}{BC}

\DeclareMathOperator{\eva}{ev}

\ifpdf
\hypersetup{
	pdftitle={Periodic Center Manifolds for Nonhyperbolic Limit Cycles in ODEs},
	pdfauthor={B. Lentjes, M. Windmolders, Yu. A. Kuznetsov}
}
\fi

\begin{document}
	\maketitle
	
	\begin{abstract}
		In this paper, we deal with a classical object, namely, a nonhyperbolic limit cycle in a system of smooth autonomous ordinary differential equations. While the existence of a center manifold near such a cycle was assumed in several studies on cycle bifurcations based on periodic normal forms, no proofs were available in the literature until recently. The main goal of this paper is to give an elementary proof of the existence of a periodic smooth locally invariant center manifold near a nonhyperbolic cycle in finite-dimensional ordinary differential equations by using the Lyapunov-Perron method. In addition, we provide several explicit examples of analytic vector fields admitting (non)-unique, (non)-$C^{\infty}$-smooth and (non)-analytic periodic center manifolds.
	\end{abstract}
	
	
	\begin{sloppypar}
		
		\section{Introduction}
		Center manifold theory is without doubt one of the most well-known and powerful techniques to study local bifurcations of dynamical systems \cite{Guckenheimer1983,Kuznetsov2023a}. In its simplest form, center manifold theory allows us to analyze the behavior of a complicated high-dimensional nonlinear dynamical system near a bifurcation by reducing the system to a low-dimensional invariant manifold, called the center manifold. 
		
		The center manifold theorem for finite-dimensional ordinary differential equations (ODEs) near a nonhyperbolic equilibrium has first been proved in \cite{Pliss1964,Kelley1967} and developed further in \cite{Hirsch1977,Carr1981}. Over the years, the existence of a center manifold near a nonhyperbolic equilibrium has been established for va\-ri\-ous other classes of dynamical systems by employing different techniques, such as, for example, the graph transform \cite{Shub1987,Sandstede2015}, the parametrization method \cite{Cabre2003,Berg2020} and the Lyapunov-Perron method \cite{Vanderbauwhede1989,Haragus2011,Eldering2013}. Mainly this last method has been proven to be very successful in the setting of infinite-dimensional dynamical systems. For example, the center manifold theorem for equilibria has been obtained under va\-ri\-ous assumptions for ODEs in Banach spaces \cite{Vanderbauwhede1987,Chow1988,Vanderbauwhede1989,Vanderbauwhede1992,Haragus2011}, partial differential equations \cite{Henry1981,Kirchgaessner1982,Mielke1986,Mielke1988,Bates1989}, stochastic dynami\-cal systems \cite{Boxler1989,Du2006,Chen2015,Chen2018,Li2022}, classical delay (differential) equations \cite{Chafee1971,Diekmann1991a,Hale1993,Diekmann1995,Bosschaert2020}, renewal equations \cite{Diekmann1991a,Diekmann1995,Diekmann2008}, abstract delay (differential) equations \cite{Janssens2020}, impulsive delay differential equations \cite{Church2021}, mixed functional differential and difference equations \cite{Hupkes2006,Hupkes2008a}. Various interesting and important qualitative properties of center manifolds for equilibria can be found in \cite{Sijbrand1985} and an extensive literature overview on such manifolds in various classes of dynamical systems can be found in \cite{Osipenko2009}. In all cases, the dimension of the center manifold is equal to the number of the critical eigenvalues of the equilibrium, i.e. those with zero real parts.
		
		The natural question arises if the whole center manifold construction can be repeated for nonhyperbolic periodic orbits (cycles) in various classes of dynamical systems. While the literature for center manifolds for equilibria is extensive, the same cannot be said for center manifolds near cycles. A first proof on the existence and smoothness of a center manifold for periodic mixed functional differential equations was given in \cite{Hupkes2008} and has been later adapted in \cite{Church2018,Church2021} to the setting of periodic impulsive delay differential equations. Recently, in \cite{Lentjes2023a}, the existence of a smooth periodic finite-dimensional center manifold near a cycle for classical delay differential equations has been established using the general sun-star calculus framework \cite{Clement1987,Clement1988,Clement1989,Clement1989a,Diekmann1991,Diekmann1995}, which expands its applicability to various other classes of delay equations. Here, the dimension of the center manifold is also determined by the number of the critical multipliers of the cycle, including the trivial  (equal to one) multiplier.
		
		However, as the state space in all mentioned references on this topic is infinite-dimensional, many proofs are rather involved as one must rely on non-trivial functional analytic techniques. While the resulting center manifold theorems could be applied to finite-dimensional ODEs without delays, this is certainly a redundant overkill. The main goal of this paper is to directly state and prove a center manifold theorem for cycles in finite-dimensional ODEs, using only elementary tools. Essentially, the proofs below are rather straightforward adaptations of those from \cite{Lentjes2023a} in a much simpler finite-dimensional context. We already remark that our exposition is based on the classical Lyapunov-Perron method as a variation of constants formula is easily available in this setting. 
		
		To study stability and bifurcations of limit cycles in ODEs, one can alternatively work with a Poincar\'{e} map on a cross-section to the cycle \cite{Guckenheimer1983,Kuznetsov2023a}. In most cases, this is sufficient, but then we miss one dimension, i.e. the phase coordinate along the cycle. It should also be noted immediately that the existence of a smooth (non-unique) center manifold for the fixed point of the Poincar\'{e} map on a cross-section to the cycle does not imply directly the existence of a smooth center manifold in a tubular neighborhood of the cycle.\footnote{Notice that one can deduce the existence of the unique stable and unstable manifold near a hyperbolic cycle from the fixed point of the Poincar\'{e} map, see \cite[Theorem 10.3.2]{Hale1993}.} A motivation in keeping this phase dimension is to directly obtain all information of the dynamics near the cycle. 
		
		We fill two important gaps in the literature. First, from a theoretical point of view, the results from \cite{Iooss1988,Iooss1999} on the existence of a special coordinate system on the center manifold that allows us to describe the local dynamics near a bifurcating cycle in terms of so-called periodic normal forms, rely heavily on the local invariance and smoothness properties of the center manifold. However, no proof nor a reference towards the literature has been provided which ensures the existence of a periodic sufficiently smooth center manifold near a nonhyperbolic cycle. Second, from a more practical point of view, many researchers use nowadays the well-known software package \verb|MatCont| \cite{Dhooge2003,Dhooge2008} to study codimension one and two bifurcations of limit cycles in finite-dimensional ODEs. In particular, if one is interested in determining the nature (subcritical, supercritical or degenerate) of a bifurcation, one should compute the critical normal form coefficients of an associated periodic normal form. However, the computation of these coefficients in \verb|MatCont| employs a combination of the periodic normalization method \cite{Kuznetsov2005,Witte2013,Witte2014}, again based on the smoothness and local invariance of the center manifold, with the special coordinate system and the periodic normal forms mentioned above. 
		
		\subsection{Statement of the main theorem}
		Let $f: \mathbb{R}^{n} \rightarrow \mathbb{R}^n$ be a $C^{k+1}$-smooth vector field for some finite $k \geq 1$ and consider the ordinary differential equation
		\begin{equation} \label{eq:ODE} \tag{ODE}
			\dot{x}(t) = f(x(t)),
		\end{equation}
		where $x(t) \in \mathbb{R}^n$. Assume that \eqref{eq:ODE} admits a $T$-periodic solution $\gamma$ for some (minimal) $T > 0$ and let $\Gamma := \gamma(\mathbb{R})$ denote the associated (limit) \emph{cycle}. Consider now the variational equation around $\Gamma$
		\begin{equation} \label{eq:variational}
			\dot{y}(t) = A(t)y(t),
		\end{equation}
		where $A(t) := Df(\gamma(t))$ and $y(t) \in \mathbb{R}^n$. The unique (global) solution $y$ of \eqref{eq:variational} is generated by the \emph{fundamental matrix} $U(t,s) \in \mathbb{R}^{n \times n }$ as $y(t) = U(t,s)y_0$ for all $(t,s) \in \mathbb{R}^2$, where $y_0 \in \mathbb{R}^n$ is an initial condition specified at time $s$. The eigenvalues of the matrix $U(s+T,s)$ are called \emph{Floquet multipliers} (of $\Gamma$), and we say that $\Gamma$ is \emph{nonhyperbolic} if there are at least $n_0 + 1 \geq 2$ Floquet multipliers on the unit circle that are counted with algebraic multiplicity. Let $E_0(s)$ denote the $(n_0+1)$-dimensional \emph{center subspace} (at time $s$) defined by the direct sum of all generalized eigenspaces with a Floquet multiplier on the unit circle and let $E_0 := \{(s,y_0) \in \mathbb{R} \times \mathbb{R}^n : y_0 \in E_0(s) \}$ denote the \emph{center bundle}. The main result on the existence of a periodic smooth local invariant center manifold near the cycle $\Gamma$ is summarized in \Cref{thm:main} and two illustrative examples of two-dimensional \emph{local center manifolds around $\Gamma$}, denoted by $\mathcal{W}_{\loc}^c(\Gamma)$, can be found in \Cref{fig:CM}. Explicit minimal equations (model systems) of the form \eqref{eq:ODE} admitting a $2\pi$-periodic two-dimensional center manifold around a nonhyperbolic cycle are given by \Cref{ex:analytic} (cylinder) and \Cref{ex:mobius} (M\"obius band).
		
		\begin{theorem}[Center Manifold Theorem for Cycles]
			\label{thm:main} 
			Consider \eqref{eq:ODE} with a $C^{k+1}$-smooth right-hand side $f : \mathbb{R}^n \to \mathbb{R}^n$ for some finite $k \geq 1$. Let $\gamma$ be a $T$-periodic solution of \eqref{eq:ODE} such that the associated cycle $\Gamma := \gamma(\mathbb{R})$ is nonhyperbolic with $(n_0 + 1)$-dimensional center subspace $E_0(s)$ at time $s \in \mathbb{R}$. Then there exists a locally defined $T$-periodic $C^k$-smooth $(n_0 + 1)$-dimensional invariant manifold $\mathcal{W}^c_{\loc}(\Gamma)$ defined around $\Gamma$ and tangent to the center bundle $E_0$.
		\end{theorem}
		
		\begin{figure}[ht]
			\centering
			\includegraphics[width = 12cm]{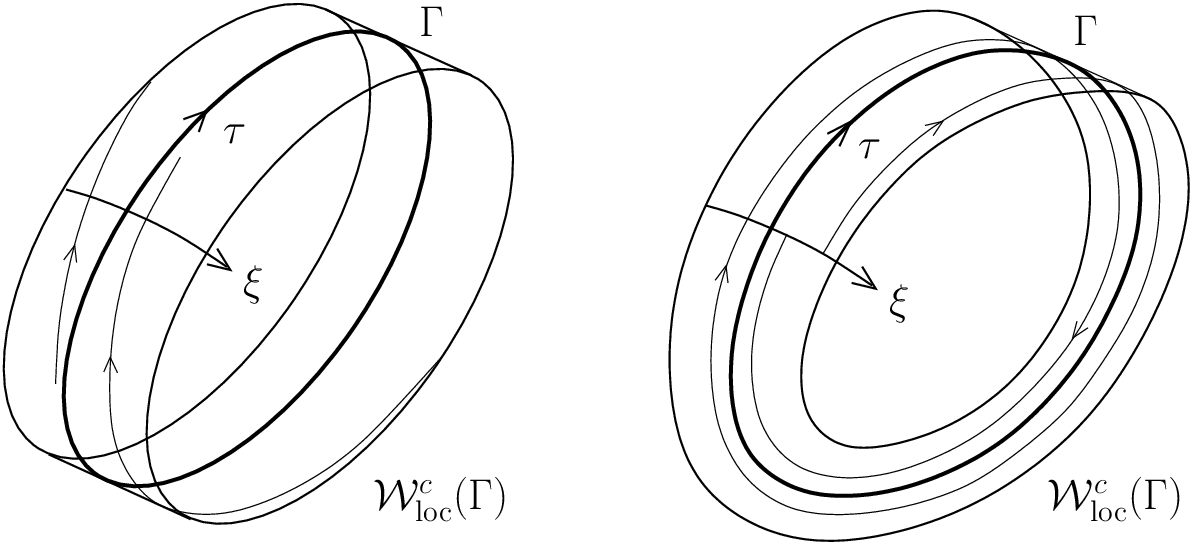}
			\caption{Illustration of two-dimensional local center manifolds around $\Gamma$ for $n=3$ \cite{Kuznetsov2023a}. The left figure represents the case when $-1$ is not a Floquet multiplier and then $\mathcal{W}_{\loc}^{c}(\Gamma)$ is locally diffeomorphic to a cylinder in a neighborhood of $\Gamma$. The right figure represents the case when $-1$ is a Floquet multiplier and then $\mathcal{W}_{\loc}^{c}(\Gamma)$ is locally diffeomorphic to a Möbius band in a neighborhood of $\Gamma$. The $(\tau,\xi)$-coordinate system on the center manifold is the special coordinate system described in \cite{Iooss1988,Iooss1999}.}
			\label{fig:CM}
		\end{figure}
		
		\subsection{Overview}
		The paper is organized as follows. In \Cref{sec:Floquet} we review some basic principles of Floquet theory for ODEs and elaborate a bit more on spectral decompositions. In \Cref{sec:existence} we use the theory from previous section to prove the existence of a Lip\-schitz continuous center manifold for (nonlinear) periodic ODEs. In \Cref{sec:properties} we prove that this center ma\-ni\-fold is periodic, sufficiently smooth, locally invariant and its tangent bundle is precisely the center bundle. The technical proofs regarding smoothness of the center manifold are relegated to Appendix A. Combining all these results proves \Cref{thm:main}. In \Cref{sec:examples} we provide explicit examples of analytic vector fields admitting (non)-unique, (non)-$C^{\infty}$-smooth and (non)-analytic periodic center manifolds.
		
		\section{Floquet theory and spectral decompositions} \label{sec:Floquet}
		Consider $\eqref{eq:ODE}$ admitting a $T$-periodic solution $\gamma$ with associated cycle $\Gamma := \gamma(\mathbb{R})$. The aim of this section is to determine the stability of $\Gamma$ and characterize nonhyperbolicity using Floquet theory. This (linear) theory will allow us to state and prove results regarding spectral properties of our operators and spaces of interest. Standard references for this entire section are the books \cite{Hale2009,Iooss1999,Kuznetsov2023a} on ODEs and \cite{Axler2014} on Linear Algebra. All unreferenced claims relating to basic properties of Floquet theory and spectral decompositions can be found here.
		
		To study the stability of $\Gamma$, set $x = \gamma + y$ and notice that $y$ satisfies the (nonlinear) periodic ODE
		\begin{equation} \label{eq:A(t) R}
			\dot{y}(t) = A(t)y(t) + R(t,y(t)),
		\end{equation}
		where $A(t) := Df(\gamma(t))$ and $R(t,y(t)) := f(\gamma(t) + y(t)) - f(\gamma(t)) - A(t)y(t)$. Hence, $A : \mathbb{R} \to \mathbb{R}^{n \times n}$ is a $T$-periodic $C^k$-smooth function and $R : \mathbb{R} \times \mathbb{R}^n \to \mathbb{R}^n$ is $T$-periodic in the first component, $C^k$-smooth, $R(\cdot,0) = 0$ and $D_2 R(\cdot,0) = 0$, i.e. $R$ consists solely of nonlinear terms. Note that the nonlinearity $R$ has one degree of smoothness less than the original vector field $f$. For a starting time $s \in \mathbb{R}$ and initial condition $y_0 \in \mathbb{R}^n$ for \eqref{eq:A(t) R}, it follows from the $C^k$-smoothness of $R$ and the Picard-Lindelöf theorem that \eqref{eq:A(t) R} admits a unique (maximal) solution for all $t \in \mathbb{R}$ sufficiently close to $s$. Hence, let for such $t$ and $s$ the map $S(t,s,\cdot) : \mathbb{R}^n \to \mathbb{R}^n$ denote the (\emph{time-dependent}) \emph{flow}, also called \emph{process} \cite{Kloeden2011}, of \eqref{eq:A(t) R}. One can verify by uniqueness of solutions that
		\begin{align} \label{eq:Sts}
			S(t,r,S(r,s,\cdot)) = S(t+r,s,\cdot), \quad  S(s,s,\cdot) = I, \quad S(t+T,s+T,\cdot) = S(t,s,\cdot),
		\end{align}
		for all $t,r \in \mathbb{R}$ sufficiently close to $s$. It is clear from this construction that studying solutions of \eqref{eq:ODE} near $\Gamma$ is equivalent to studying solutions of \eqref{eq:A(t) R} near the origin. Therefore, we start by investigating the linearization of \eqref{eq:A(t) R} around the origin:
		\begin{equation} \label{eq:A(t)y(t)}
			\dot{y}(t) = A(t)y(t).
		\end{equation}
		Observe that its (global) solutions are generated by the \emph{fundamental matrix} $U(t,s) \in \mathbb{R}^{n \times n}$ as $y(t) = U(t,s)y_0$ for all $(t,s) \in \mathbb{R}^2$, whenever an initial condition $y_0 \in \mathbb{R}^n$ at starting time $s \in \mathbb{R}$ is specified. Moreover, we have that the map $(t,s) \mapsto U(t,s)$ is $C^k$-smooth. Using uniqueness of solutions for \eqref{eq:A(t)y(t)}, $T$-periodicity of $A$, and the fact that $U(s,s) = I$ for all $s \in \mathbb{R}$, one can easily verify that
		\begin{align} \label{eq:propsU(t,s)}
			U(t,r)U(r,s) = U(t,s), \quad U(t,s)^{-1} = U(s,t), \quad U(t+T,s+T) = U(t,s), 
		\end{align}
		for all $t,r,s \in \mathbb{R}$.
		\begin{lemma} \label{lemma:partialUts}
			There holds for any $(t,s) \in \mathbb{R}^2$ that
			\begin{align*}
				\frac{\partial}{\partial t} U(t,s) = A(t)U(t,s), \quad \frac{\partial}{\partial t} U(s,t) = -U(s,t)A(t).
			\end{align*}
		\end{lemma}
		\begin{proof}
			The first equality follows immediately from \eqref{eq:A(t)y(t)}. To prove the second equality, observe from \eqref{eq:propsU(t,s)} that $U(s,t)U(t,s) = I$ and so differentiating both sides with respect to $t$ yields after rearranging
			\begin{equation*}
				\bigg(\frac{\partial}{\partial t}U(s, t)\bigg)U(t, s) = -U(s, t)A(t)U(t, s),
			\end{equation*}
			which proves the claim.
		\end{proof}
		Combining the first and third equality of \eqref{eq:propsU(t,s)} together with induction, one proves that $U(s+kT,s) = U(s+T,s)^k$ for all $s \in \mathbb{R}$ and $k \in \mathbb{Z}$ and so $y(s+kT) = U(s+T,s)^k y_0$. Hence, the long term behavior of the solution $y$ is determined by the \emph{monodromy matrix} (at time $s$) $U(s+T,s)$ and especially its eigenvalues, called \emph{Floquet multipliers}. To develop a spectral theory for our problem of interest, notice that one has to \emph{complexify} the state space $\mathbb{R}^n$ and all discussed operators defined on $\mathbb{R}^n$, i.e. one has to extend the state space to $\mathbb{C}^n$ and extend all discussed operators to $\mathbb{C}^n$, see \cite[Chapter 9]{Axler2014} for more information. However, for the sake of simplicity, we will not introduce any additional notation for the complexification of the operators.
		
		Let us now study the spectrum $\sigma(U(s+T,s))$ of $U(s+T,s)$ in depth. It follows from, e.g., \cite[Theorem 1.6]{Kuznetsov2023a} that the Floquet multipliers are independent of the starting time $s$ and that $1$ is always a Floquet multiplier. To see this last claim, differentiating $\dot{\gamma}(t) = f(\gamma(t))$ yields $\ddot{\gamma}(t) = A(t)\dot{\gamma}(t)$ and so $\dot{\gamma}$ is a solution of \eqref{eq:A(t)y(t)}, i.e. $\dot{\gamma}(t) = U(t,s)\dot{\gamma}(s)$. Exploiting $T$-periodicity of $\gamma$ yields $\dot{\gamma}(s) = \dot{\gamma}(s+T) = U(s+T,s)\dot{\gamma}(s)$, which proves that $1$ is an eigenvalue of $U(s+T,s)$ with associated eigenvector $\dot{\gamma}(s)$.
		
		Let $\lambda$ be a Floquet multiplier of algebraic multiplicity $m_\lambda$, i.e. the $m_\lambda$-dimensional $U(s+T,s)$-invariant subspace $E_\lambda(s) := \ker((U(s+T,s) - \lambda I)^{m_\lambda})$ of $\mathbb{C}^n$ is maximal, or equivalently, $m_\lambda$ is the order of the root $\mu=\lambda$ of the characteristic polynomial $\det(U(s+T,s) - \mu I)$. This allows us to choose a basis of $m_\lambda$ linearly independent (generalized) eigenvectors $\zeta_1(s),\dots,\zeta_{m_\lambda}(s)$ of $E_\lambda(s)$. Moreover, let $\pi_\lambda(s)$ be the projection from $\mathbb{C}^n$ to $E_\lambda(s)$ along the direct sum $\oplus_{\mu \neq \lambda} E_{\mu}(s)$. Our next aim is to extend $\zeta_1(s),\dots,\zeta_{m_\lambda}(s)$ and $\pi_\lambda(s)$ forward and backward in time. The results can be found in the following two lemmas.
		
		\begin{lemma} \label{lemma:basis}
			Let $\lambda$ be a Floquet multiplier, then the restriction $U_\lambda(t,s) : E_\lambda(s) \to E_\lambda(t)$ is well-defined and invertible for all $(t,s) \in \mathbb{R}^2$. Moreover, there exist $C^k$-smooth maps $\zeta_i : \mathbb{R} \to \mathbb{C}^n$ such that $\zeta_1(t),\dots,\zeta_{m_\lambda}(t)$ is a basis of $E_\lambda(t)$ for all $t \in \mathbb{R}$.
		\end{lemma}
		\begin{proof}
			One can verify easily from the equalities in \eqref{eq:propsU(t,s)} that
			\begin{equation*}
				(U(t+T,t) - \lambda I)U(t,s) = U(t,s)(U(s+T,s) - \lambda I).
			\end{equation*}
			Therefore, for each $v \in E_\lambda(s)$, we get
			\begin{align*}
				(U(t+T,t) - \lambda I)^{m_{\lambda}} U(t,s)v = U(t,s)(U(s+T,s) - \lambda I)^{m_{\lambda}} v = 0,
			\end{align*}
			which proves that $U_\lambda(t,s)v \in E_\lambda(t)$ since the Floquet multipliers are independent of the starting time. As $U(t,s)$ is invertible, its restriction $U_\lambda(t,s)$ is invertible as well and this proves the first claim. To prove the second claim, let $\zeta_1(s),\dots,\zeta_{m_\lambda}(s)$ be a basis of $E_\lambda(s)$ and define for all $i \in \{1,\dots,m_\lambda\}$ the $C^k$-smooth maps $\zeta_i : \mathbb{R} \to \mathbb{C}^n$ by $\zeta_i(t) := U_\lambda(t,s)\zeta_i(s)$. By the first claim, it is clear that $\zeta_1(t),\dots,\zeta_{m_\lambda}(t)$ is a basis of $E_\lambda(t)$ for all $t\in \mathbb{R}$, and this completes the proof. 
		\end{proof}
		
		\begin{lemma}[{\cite[Proposition III.2]{Iooss1999}}] \label{lemma:Iooss}
			Let $\lambda$ be a Floquet multiplier, then there exists a $T$-periodic $C^k$-smooth map $\pi_\lambda : \mathbb{R} \to \mathbb{C}^{n \times n}$ such that  $\pi_\lambda(t)$ is the projection from $\mathbb{C}^n$ onto $E_\lambda(t)$ for all $t \in \mathbb{R}$ and satisfies the periodic linear ODE
			\begin{equation} \label{eq:ODEpi}
				\dot{\pi}_\lambda(t) = A(t)\pi_\lambda(t) - \pi_\lambda(t)A(t).
			\end{equation}
		\end{lemma}
		It will be convenient in the sequel to introduce the sets $\Lambda_{-} := \{ \lambda \in \sigma(U(s+T,s)) : |\lambda| < 1 \}, \Lambda_{0} := \{ \lambda \in \sigma(U(s+T,s)) : |\lambda| = 1 \}$ and $\Lambda_{+} := \{ \lambda \in \sigma(U(s+T,s)) : |\lambda| > 1 \}$, where the elements of these sets have to be counted with algebraic multiplicity. We say that the cycle $\Gamma$ is \emph{nonhyperbolic} if there are at least $n_0 + 1 \geq 2$ Floquet multipliers on the unit circle that are counted with algebraic multiplicity, i.e. the cardinality of $\Lambda_0$ is at least $2$.
		
		\begin{proposition} \label{prop:criteria}
			The following properties hold.
			\begin{enumerate}
				\item[$1.$] For each $s \in \mathbb{R}$, the Euclidean $n$-space admits a direct sum decomposition
				\begin{equation} \label{eq:decomposition}
					\mathbb{R}^n = E_{-}(s) \oplus E_0(s) \oplus E_{+}(s)
				\end{equation}
				in a \emph{stable subspace}, \emph{center subspace}, and \emph{unstable subspace} $($at time $s)$, respectively.
				\item[$2.$] There exist three $T$-periodic $C^k$-smooth projectors $\pi_{i} : \mathbb{R} \to \mathbb{R}^{n \times n}$ with $\ran(\pi_i(s))= E_i(s)$ for all $s \in \mathbb{R}$ and $i \in \{-,0,+\}$. 
				\item[$3.$] There exists a constant $N \geq 0$ such that $\sup_{s \in \mathbb{R}}(\|\pi_{-}(s)\| + \|\pi_{0}(s)\| + \|\pi_{+}(s)\|) = N < \infty$.
				\item[$4.$] The projections satisfy: $\pi_{i}(s)\pi_j(s) = 0$ for all $s \in \mathbb{R}$ and $i \neq j$ with $i,j \in \{-,0,+\}$.
				\item[$5.$] The projections commute with the fundamental matrix: $U(t,s)\pi_i(s) = \pi_i(t)U(t,s)$ for all $(t,s) \in \mathbb{R}^2$ and $i \in \{-,0,+\}$.
				\item[$6.$] The restrictions $U_{i}(t,s) : E_{i}(s) \to E_{i}(t)$ are well-defined and invertible for all $(t,s) \in \mathbb{R}^2$ and $i \in \{-,0,+\}$.
				\item[$7.$] The decomposition \eqref{eq:decomposition} is an exponential trichotomy on $\mathbb{R}$ meaning that there exist $a < 0 < b$ such that for every $\varepsilon > 0$ there exists a $C_\varepsilon > 0$ such that
				\begin{align*}
					\|U_-(t,s)\| &\leq C_\varepsilon e^{a(t-s)}, \quad t \geq s,\\
					\|U_0(t,s)\| &\leq C_\varepsilon e^{\varepsilon|t-s|}, \quad t,s \in \mathbb{R},\\
					\|U_+(t,s)\| &\leq C_\varepsilon e^{b(t-s)}, \quad t \leq s.
				\end{align*}
			\end{enumerate}
		\end{proposition}
		\begin{proof}
			We verify the seven properties step by step.
			
			1. From the generalized eigenspace decomposition theorem, we have that
			\begin{equation*}
				\mathbb{R}^n = \bigoplus_{\lambda \in \sigma(U(s + T,s))}E_\lambda(s),
			\end{equation*}
			and if we define $E_{i}(s) := \oplus_{\lambda \in \Lambda_i} E_\lambda(s)$ for $i \in \{-,0,+\}$, the result follows. Notice that $E_i(s)$ can be regarded as real vector space since, if $\lambda \in \Lambda_i$, then $\overline{\lambda} \in \Lambda_{i}$ because $U$ is a real operator. 
			
			2. Define for $i \in \{-,0,+\}$ the map $\pi_{i}$ by $\pi_i(s) := \sum_{\lambda \in \Lambda_i} \pi_\lambda(s)$. It follows from linearity and \Cref{lemma:Iooss} that $\pi_i$ is $T$-periodic and $C^k$-smooth. By construction, the range of $\pi_i(s)$ is $E_{i}(s)$ for all $s \in \mathbb{R}$. The same argument as in the first assertion shows that $\pi_i$ is a real operator.
			
			3. Because $\pi_i$ and the norm $\| \cdot \|$ are continuous, we have that the map $\| \pi_i (\cdot) \| : \mathbb{R} \to \mathbb{R}$ is $T$-periodic and continuous. The claim follows now from applying three times the extreme value theorem.
			
			4. For $y \in \mathbb{R}^n$ the direct sum \eqref{eq:decomposition} admits a unique decomposition $y = y_- + y_0 + y_+$ with $y_i \in E_i(s)$. Hence, $\pi_i(s)\pi_j(s)y = \pi_i(s)y_j = 0$ if $i \neq j$ for all $s \in \mathbb{R}$.
			
			5. Differentiating $t \mapsto U(t,s)\pi_i(s)$ and $t \mapsto \pi_i(t)U(t,s)$ while using \eqref{eq:ODEpi}, one sees that they both satisfy \eqref{eq:A(t)y(t)}. Since they coincide at time $t = s$, we have by uniqueness that they must be equal.
			
			6. Define for $i \in \{-,0,+\}$ the map $U_{i}(t,s)$ by $U_i(t,s) := \oplus_{\lambda \in \Lambda_i} U_\lambda(t,s)$ for all $(t,s) \in \mathbb{R}^2$. The claim follows now from linearity and \Cref{lemma:basis}.
			
			7. We will only prove the $U_{-}(t,s)$ estimate as the other ones can be proven similarly. Since the spectrum of $U_-(s + T,s)$ lies inside the unit disk, it follows from the spectral radius formula that 
			\begin{equation*}
				\lim_{m \rightarrow \infty}\|U_-(s + T,s)^m\|^{\frac{1}{m}} = \max_{\lambda \in \sigma(U_-(s + T,s))}|\lambda| < 1 .
			\end{equation*}
			Hence, there exists an $a < 0$ and an integer $m > 0$ such that
			\begin{equation*}
				\|U_-(s + T,s)^m\| < (1 + aT)^m,
			\end{equation*}
			and by continuity of the map $t \mapsto U_-(t,s)$, there is some $L > 0$ such that $\sup_{s \leq t \leq s + T}\|U_-(t,s)\| \leq L$. Denote $L_m := L \max_{j = 0, \dots, m - 1}\|U_-(s + T,s)^j\|$ and let $m_t$ be the largest nonnegative integer such that $s + m_tmT \leq t$ and let $0 \leq m_t^\star \leq m - 1$ be the largest integer such that $s + m_tmT + m_t^\star T \leq t$. Using \eqref{eq:propsU(t,s)}, one obtains
			\begin{align*}
				U_-(t,s) = U_-(t - m_tmT - m_t^\star T, s) U_-(s + T,s)^{m_t^\star}U_-(s + T,s)^{m_tm}.
			\end{align*}
			By the maximum property of $m_t^\star$: $s \leq t - m_tmT - m_t^\star T \leq s + m_tmT + (m_t^\star + 1)T - m_tmT - m_t^{\star} T = s + T$, and so
			\begin{align*}
				\|U_-(t,s)\| \leq L_n\|U_-(s + T,s)^m\|^{m_t} \leq L_m (1 + aT)^{m_tm} \leq L_m [(1 + aT)^{\frac{1}{aT}}]^{a(t - s)} \leq L_m e^{a(t - s)},
			\end{align*}
			where we used in the last equality the fact that the map $x \mapsto (1+\frac{1}{x})^{x}$ is monotonically increasing on $(-\infty,0)$ and converges to $e$ as $x \to -\infty$. For the other estimates, for a given $\varepsilon > 0$ and sufficiently large $m' \in \mathbb{N}$, one finds that there exists a $M_\varepsilon$ and $N_{m'}$ such that $||U_0(t,s)|| \leq M_\varepsilon e^{\varepsilon|t-s|}$ for $(t,s) \in \mathbb{R}^2$ and $||U_{+}(t,s)|| \leq N_{m'} e^{b(t-s)}$ for $t \leq s$. Choosing $C_\varepsilon := \max \{L_m, M_\varepsilon, N_{m'}\}$ proves the claim.
		\end{proof}
		
		\section{Existence of a Lipschitz center manifold} \label{sec:existence}
		The aim of this section is to prove the existence of a (local) center manifold for \eqref{eq:A(t) R} around the origin. The proof consists of four steps. In the first step, we show that we can formulate \eqref{eq:A(t) R} equivalently as an integral equation. In the second step, we determine a pseudo-inverse for solutions of this integral equation on a suitable Banach space. In the third step, we modify our nonlinearity $R$ outside a ball of radius $\delta > 0$ such that it becomes Lipschitz con\-ti\-nu\-ous, and eventually a contraction when $\delta$ is chosen small enough. In the last step, we construct a (family of) fixed point operators using the pseudo-inverse and modified nonlinearity for a sufficiently small $\delta$. The fixed points of these contractions constitute the center manifold.

		\begin{lemma} \label{lemma:integralequivalence}
			The ordinary differential equation \eqref{eq:A(t) R} is equivalent to the integral equation
			\begin{equation} \label{eq:integraleq}
				u(t) = U(t,s)u(s) + \int_s^t U(t,\tau)R(\tau,u(\tau)) d\tau.
			\end{equation}
		\end{lemma}
		\begin{proof}
			Any $u$ satisfying \eqref{eq:integraleq} is clearly differentiable, and it follows from the Leibniz integral rule that 
			\begin{align*}
				\dot{u}(t) &= A(t)U(t, s)u(s) + U(t,t)R(t, u(t)) + \int_s^tA(t)U(t, \tau)R(\tau, u(\tau))d\tau \\
				&= A(t)\left[U(t, s)u(s) + \int_s^tU(t, \tau)R(\tau, u(\tau))d\tau\right] +  R(t, u(t)) = A(t)u(t) + R(t, u(t)),  
			\end{align*}
			which proves that $u$ satisfies \eqref{eq:A(t) R}. Conversely, let $u$ satisfy \eqref{eq:A(t) R} and let $w(t) = U(s, t)u(t)$. Then
			\begin{equation*}
				\dot{w}(t) = \bigg(\frac{\partial}{\partial t}U(s, t)\bigg)u(t) + U(s, t)\dot{u}(t).
			\end{equation*}
			The second equality in \Cref{lemma:partialUts} together with the fact that $u$ satisfies \eqref{eq:A(t) R} shows that $\dot{w}(t) = U(s,t)R(t,u(t))$. Integrating both sides with respect to $t$ yields
			\begin{align*}
				u(t) &= U(t, s)u(s) + U(t, s)\int_s^tU(s, \tau)R(\tau, u(\tau))d\tau = U(t, s)u(s) + \int_s^tU(t, \tau)R(\tau, u(\tau))d\tau,
			\end{align*}
			where we used \eqref{eq:propsU(t,s)} in the last equality.
		\end{proof}
		
		Let $C_b(\mathbb{R}, \mathbb{R}^{n})$ denote the Banach space of $\mathbb{R}^n$-valued continuous bounded functions defined on $\mathbb{R}$ equipped with the supremum norm $\| \cdot \|_{\infty}$. If we want to study solutions of \eqref{eq:A(t)y(t)} (or equivalently \eqref{eq:integraleq} with $R = 0$) in the center subspace, it turns that such solutions can be unbounded, and so we can not work in the space $C_b(\mathbb{R}, \mathbb{R}^{n})$. Instead, we must work in a function space that allows limited (sub)exponential growth both at plus and minus infinity. Therefore, define for any $\eta,s \in \mathbb{R}$ the normed space
		\begin{equation*}
			\BC_{s}^{\eta} := \bigg \{ f \in C(\mathbb{R},\mathbb{R}^n) : \sup_{t \in \mathbb{R}} e^{-\eta|t-s|}\|f(t)\| < \infty \bigg \},
		\end{equation*}
		with the weighted supremum norm
		\begin{equation*}
			\|f\|_{\eta,s} := \sup_{t \in \mathbb{R}} e^{-\eta|t-s|}\|f(t)\|.
		\end{equation*}
		Since the linear map $\iota : (C_b(\mathbb{R}, \mathbb{R}^{n}), \|\cdot\|_{\infty}) \rightarrow (\BC_{s}^{\eta}, \|\cdot\|_{\eta,s})$ defined by $\iota(f)(t) := e^{\eta|t - s|}f(t)$ is an isometry, it is clear that $(\BC_{s}^{\eta}, \|\cdot\|_{\eta,s})$ is a Banach space. The following result proves that all solutions of \eqref{eq:A(t)y(t)} on the center subspace belong to $\BC_{s}^{\eta}$.
		
		\begin{proposition} \label{prop:X0s}
			Let $\eta \in (0,\min \{-a,b\})$ and $s \in \mathbb{R}$. Then
			\begin{align*}
				E_0(s) = \{y_0 \in \mathbb{R}^n : &\text{ there exists a solution of \eqref{eq:A(t)y(t)}} \text{ through $y_0$ belonging to $\BC_{s}^{\eta}$}\}.
			\end{align*}
		\end{proposition}
		\begin{proof}
			Choose $y_0 \in E_0(s)$ and define $y(t) = U_0(t,s)y_0$, which is indeed a solution of \eqref{eq:A(t)y(t)} through $y_0$. Let $ \varepsilon \in (0,\eta]$ be given. The exponential trichotomy from \Cref{prop:criteria} shows that
			\begin{align*}
				e^{-\eta|t - s|}\|y(t)\| &= e^{-\eta|t - s|}\|U_0(t,s)y_0\| 
				\leq C_\varepsilon e^{(\varepsilon - \eta)|t - s|}\|y_0\| 
				\leq C_\varepsilon\|y_0\|, \quad \forall t,s \in \mathbb{R}.    
			\end{align*}
			Taking the supremum over $t \in \mathbb{R}$ yields $y \in \BC_s^\eta$. Conversely, let $y_0 \in \mathbb{R}^n$ be such that $y$, defined by $y(t) = U(t,s)y_0$, is in $\BC_s^\eta$. For $t \geq \max\{s,0\}$ and $\varepsilon \in (0,\eta]$, we get
			\begin{align*}
				\|\pi_+(s)y_0\| &= \|U_+(s,t)\pi_+(t)y(t)\| 
				\leq C_\varepsilon e^{b(s - t)}N\|y(t)\|,
			\end{align*}
			which shows that
			\begin{equation*}
				e^{-\eta |t-s|}\|y(t)\| \geq \frac{e^{(b-\eta)(t-s)}}{C_\varepsilon N} \|\pi_+(s)y_0\| \to \infty,
			\end{equation*}
			as $t \to \infty$ unless $\pi_+(s)y_0 = 0$ as $y \in \BC_s^\eta$. Similarly, one can prove that $\pi_-(s)y_0 = 0$ and so $y_0 = (\pi_-(s) + \pi_0(s) + \pi_+(s))y_0 = \pi_0(s)y_0$, i.e. $y_0 \in E_0(s)$.
		\end{proof}

		\subsection{Bounded solutions of the linear inhomogeneous equation} \label{subsec: bounded solutions}
		Let $f : \mathbb{R} \to \mathbb{R}^n$ be a continuous function and consider the linear inhomogeneous integral equation
		\begin{equation} \label{eq:inhomogeneous CMT}
			u(t) = U(t,s)u(s) + \int_s^t U(t,\tau)f(\tau) d\tau,
		\end{equation}
		for all $(t,s) \in \mathbb{R}^2$. To prove existence of a center manifold, we need a pseudo-inverse of (exponentially) bounded solutions of \eqref{eq:inhomogeneous CMT}. To do this, define (formally) for any $\eta \in (0, \min \{-a,b \})$ and $s \in \mathbb{R}$ the operator $\mathcal{K}_{s}^{\eta} : \BC_s^\eta \to \BC_s^\eta$ as
		\begin{align*}
			(\mathcal{K}_{s}^{\eta} f)(t) &:= \int_s^tU(t,\tau)\pi_0(\tau)f(\tau)d\tau 
			+ \int_{\infty}^tU(t,\tau)\pi_+(\tau)f(\tau)d\tau
			+ \int_{-\infty}^tU(t,\tau)\pi_-(\tau)f(\tau)d\tau,
		\end{align*}
		and check that this operator is well-defined. This will be proven in the following proposition and also the fact that $\mathcal{K}_{s}^{\eta}$ is precisely the pseudo-inverse we are looking for.
		
		\begin{proposition} \label{prop:ketas}
			Let $\eta \in (0, \min \{-a,b \})$ and $s \in \mathbb{R}$. The following properties hold.
			\begin{enumerate}
				\item[$1.$] $\mathcal{K}_{s}^{\eta}$ is a well-defined bounded linear operator. Moreover, the operator norm $\|\mathcal{K}_{s}^{\eta}\|$ is bounded above independent of $s$.
				\item[$2.$] $\mathcal{K}_{s}^{\eta}f$ is the unique solution of \eqref{eq:inhomogeneous CMT} in $\BC^{\eta}_s$ with vanishing $E_0(s)$-component at time $s$.
			\end{enumerate}
		\end{proposition}
		\begin{proof}
			We start by proving the first assertion. Clearly $\mathcal{K}_{s}^{\eta}$ is linear. Let $\varepsilon \in (0,\eta]$ be given and notice that for a given $f \in \BC^{\eta}_s$, we can write as $\mathcal{K}_{s}^{\eta}f$ as the sum of three integrals, i.e. $ \mathcal{K}_{s}^{\eta}f = I_0(\cdot,s) + I_+ + I_-$. We estimate now the norm of each integral step by step.
			
			$I_0(\cdot,s)$: The straightforward estimate
			\begin{equation*}
				\|I_0(t,s)\| \leq C_\varepsilon N \|f\|_{\eta,s} \frac{e^{\eta |t-s|}}{\eta - \varepsilon} < \infty, \quad \forall t \in \mathbb{R},
			\end{equation*}
			implies that the norm of $I_0(\cdot,s)$ is bounded above.
			
			$I_{+}:$ Notice that
			\begin{equation*}
				\|I_+(t)\| \leq C_\varepsilon N \|f\|_{\eta,s}e^{bt} \int_t^\infty e^{-b \tau + \eta |\tau - s|} d\tau, \forall t \in \mathbb{R},
			\end{equation*}
			and to prove norm boundedness of $I_+$, we have to evaluate the integral above. A calculation shows that
			\begin{align} \label{eq:expo integral}
				\int_t^\infty e^{-b \tau + \eta |\tau - s |} d\tau =    
				\begin{dcases}
					\frac{e^{-bt}}{b- \eta} e^{\eta(t-s)}, \quad & t \geq s \\
					\frac{e^{-bt}}{b+\eta}e^{\eta (s-t)} - \frac{e^{-bs}}{b+\eta} + \frac{e^{-bs}}{b-\eta}, \quad &t \leq s.
				\end{dcases}
			\end{align}
			We want to estimate the $t\leq s$ case. Notice that for real numbers $ \alpha \geq \beta$ we have
			\begin{equation*}
				(\alpha - \beta) \bigg( \frac{1}{b+\eta} - \frac{1}{b-\eta} \bigg) = \frac{-2 \eta (\alpha - \beta)}{(b+ \eta)(b-\eta)} \leq 0,
			\end{equation*}
			since $\eta < b$ by assumption. Hence,
			\begin{equation*}
				\frac{\alpha}{b+\eta} + \frac{\beta}{b-\eta} \leq \frac{\alpha}{b-\eta} + \frac{\beta}{b+\eta}.
			\end{equation*}
			We want to replace $\alpha$ by $e^{-b t + \eta s - \eta t}$ and $\beta$ by $e^{-bs}$ and therefore we have to show that $-bt + \eta s - \eta t + bs \geq 0$ which is true because $-bt + \eta s - \eta t + bs = (s-t)(b + \eta) \geq 0$ since $s-t \geq 0$. Filling this into \eqref{eq:expo integral} yields
			\begin{equation*}
				\int_t^\infty e^{-b \tau + \eta |\tau - s |} d\tau \leq \frac{e^{-bt}}{b-\eta} e^{\eta|t-s|}, 
			\end{equation*}
			which shows that
			\begin{equation*}
				\|I_{+}(t)\| \leq C_\varepsilon N \|f\|_{\eta,s}\frac{e^{\eta |t-s|}}{b-\eta} < \infty, \quad \forall t \in \mathbb{R}.
			\end{equation*}
			and so we conclude that $I_{+}$ is well-defined.
			
			$I_{-}:$ A similar estimate as for the $I_{+}$-case shows that
			\begin{equation*}
				\|I_-(t)\| \leq C_\varepsilon N \|f\|_{\eta,s}\frac{e^{\eta|t - s|}}{-a - \eta} < \infty, \quad \forall t \in \mathbb{R},
			\end{equation*}
			and so it follows that the operator norm
			\begin{equation*}
				\|\mathcal{K}_{s}^{\eta}\| \leq C_\varepsilon N \left(\frac{1}{\eta - \varepsilon} + \frac{1}{b - \eta} + \frac{1}{-a - \eta}\right) < \infty,
			\end{equation*}
			is bounded above independent of $s$. We conclude that $\mathcal{K}_{s}^{\eta}$ is a bounded linear operator on $\BC_{s}^\eta$.
			
			Let us now prove the second assertion by first showing that $\mathcal{K}_s^\eta$ is indeed a solution of \eqref{eq:inhomogeneous CMT}. Let $f \in \BC_s^\eta$ and set $u = \mathcal{K}_s^\eta f$. Then, a straightforward computation shows that
			\begin{equation*}
				U(t,s)u(s) + \int_s^t U(t,\tau)f(\tau) d\tau  = u(t),
			\end{equation*}
			and so $u$ is indeed a solution of \eqref{eq:inhomogeneous CMT}. 
			Let us now prove that $u$ has vanishing $E_0(s)$-component at time $s$, i.e. $\pi_0(s)u(s) = 0$. The mutual orthogonality of the projectors (\Cref{prop:criteria}) implies
			\begin{align*}
				\pi_0(s)u(s) = \int_\infty^s U(s,\tau ) \pi_0(\tau)\pi_{+}(\tau)f(\tau) d\tau + \int_{-\infty}^s  U(s,\tau) \pi_0(\tau) \pi_{-}(\tau)f(\tau) d\tau = 0.
			\end{align*}
			It only remains to show that $u$ is the unique solution of \eqref{eq:inhomogeneous CMT} in $\BC_s^\eta$. Let $v \in \BC_s^\eta$ be another solution of \eqref{eq:inhomogeneous CMT} with vanishing $E_0(s)$-component at time $s$. Then the function $w := u - v$ is an element of $\BC_s^\eta$ and satisfies $w(t) = U(t,s)w(s)$ for all $(t,s) \in \mathbb{R}^2$. \Cref{prop:X0s} shows us that $w(s) \in E_0(s)$ and notice that $\pi_0(s)w(s) = 0$ since $u$ and $v$ have both vanishing $E_0(s)$-component at time $s$. From \Cref{prop:criteria} we know that $w(t) = U_0(t,s)w(s)$ is in $E_0(t)$ and
			\begin{align*}
				\pi_0(t)w(t) &= \pi_0(t)U_0(t,s)w(s) = U_0(t,s)\pi_0(s)w(s) = 0,
			\end{align*}
			so $u = v$.
		\end{proof}
		
		\subsection{Modification of the nonlinearity}
		\label{subsec: modification nonlinearity}
		To prove the existence of a center manifold, a key step will be to use the Banach fixed point theorem on some specific fixed point operator. This operator we will be of course linked to the inhomogeneous equation \eqref{eq:inhomogeneous CMT}. However, we can not expect that any nonlinear operator $R(t,\cdot) : \mathbb{R}^n \to \mathbb{R}^n$ for fixed $t \in \mathbb{R}$ will impose a Lipschitz condition on the fixed point operator that will be constructed. As we are only interested in the local behavior of solutions near the origin of \eqref{eq:A(t) R}, we can modify the nonlinearity $R(t,\cdot)$ outside a ball of radius $\delta > 0$ such that eventually the fixed point operator will become a contraction. To modify this nonlinearity, introduce a $C^{\infty}$-smooth cut-off function $\xi : [0,\infty) \to \mathbb{R}$ as
		\begin{equation*}
			\xi(s) \in
			\begin{cases}
				\{ 1 \}, \quad &0 \leq s \leq 1, \\
				[0,1], \quad &0 \leq s \leq 2,\\
				\{ 0 \}, \quad & s \geq 2,
			\end{cases}
		\end{equation*}
		and define for any $\delta > 0$ the \emph{$\delta$-modification} of $R$ as the operator $R_{\delta} : \mathbb{R} \times \mathbb{R}^n \to \mathbb{R}^n$ with action
		\begin{align*}
			R_{\delta}(t,u) := R(t,u) \xi \bigg( \frac{\|\pi_0(t)u\|}{N \delta} \bigg) \xi \bigg( \frac{\|(\pi_{-}(t) + \pi_{+}(t))u\|}{N \delta} \bigg).
		\end{align*}
		Since $R$ is of the class $C^k$, the cut-off function $\xi$ is $C^{\infty}$-smooth, the Euclidean norm $\| \cdot \|$ is $C^{\infty}$-smooth on $\mathbb{R}^n \setminus \{ 0 \}$ and the projectors $\pi_-,\pi_0,\pi_{+}$ are $C^k$-smooth (\Cref{prop:criteria}), it is clear that $R_\delta$ is $C^k$-smooth. This $\delta$-modification of $R$ will ensure that the nonlinearity becomes eventually globally Lipschitz, as will be proven in the upcoming two statements.
		
		\begin{lemma} \label{lem:local lipschitz}
			There exist a $\delta_1 > 0$ and $l:[0,\delta_1] \rightarrow [0,\infty)$, continuous at $0$, such that $l(0) = 0$ and $l(\delta) =: l_\delta$ is a Lipschitz constant for $R(t,\cdot)$ on the open ball $B(0,\delta)$ for every $t \in \mathbb{R}$ and $ \delta \in (0,\delta_1]$.
		\end{lemma}
		
		\begin{proof}
			Recall that $R$ is of the class $C^k$ and that $R(t,0) = D_2R(t,0) = 0$ for all $t \in \mathbb{R}$. By continuity, choose $\delta_1 > 0$ such that $\sup\{\|D_2R(t,y)\| : y \in B(0,\delta_1)\} \leq 1$ and define the map $l$ as
			\begin{equation*}
				l(\delta) := 
				\begin{dcases}
					0, &\delta = 0, \\
					\sup\{\|D_2R(t,y)\| : y \in B(0,\delta)\}, &\delta \in (0,\delta_1].
				\end{dcases}
			\end{equation*}
			By the mean value theorem, $l(\delta)$ is a Lipschitz constant for $R(t,\cdot)$ on $B(0,\delta)$. Moreover, $l$ is monotonically increasing and observe that for each $\varepsilon > 0$, there exists a $0 < \delta_\varepsilon \leq \delta_1$ such that $\sup\{\|D_2R(t,y)\| : y \in B(0,\delta_\varepsilon)\} \leq \varepsilon$. Then for $0 < \delta \leq \delta_\varepsilon$ we have that $0 \leq l(\delta) \leq l(\delta_\varepsilon) \leq \varepsilon$ and so the map $l$ is continuous at zero.
		\end{proof}
		
		\begin{proposition} \label{prop:global lipschitz}
			For $\delta > 0$ sufficiently small, $R_{\delta}(t,\cdot)$ is globally Lipschitz continuous for all $t \in \mathbb{R}$ with Lipschitz constant $L_\delta \rightarrow 0$ as $\delta \downarrow 0$.
		\end{proposition}
		
		\begin{proof}
			Define for any $\delta > 0$ and $t \in \mathbb{R}$ the maps $\xi_\delta, \Xi_{\delta,t} : \mathbb{R}^n \to \mathbb{R}$ by
			\begin{align*}
				\xi_\delta(y) := \xi\left(\frac{\|y\|}{N\delta}\right), \quad 
				\Xi_{\delta,t}(y) := \xi_\delta(\pi_0(t)y)\xi_\delta(\pi_-(t)y + \pi_+(t)y),
			\end{align*}
			and so $R_\delta(t,y) = \Xi_{\delta,t}(y)R(t,y)$. Note that $\xi_\delta, \Xi_{\delta,t} \leq 1$ and let $C \geq 0$ be a global Lipschitz constant of $\xi$. Then by composition of Lipschitz functions, $\xi_\delta$ has a global Lipschitz constant $C/{N\delta}$. For $y,z \in \mathbb{R}^n$:
			\begin{align*}
				|\Xi_{\delta,t}(y) - \Xi_{\delta,t}(z)| &= |[\xi_\delta(\pi_0(t)y)\xi_\delta(\pi_-(t)y + \pi_+(t)y) - \xi_\delta(\pi_0(t)y)\xi_\delta(\pi_-(t)z + \pi_+(t)z)] \\
				& - [\xi_\delta(\pi_0(t)z)\xi_\delta(\pi_-(t)z + \pi_+(t)z) - \xi_\delta(\pi_0(t)y)\xi_\delta(\pi_-(t)z + \pi_+(t)z)]| \\
				&\leq \xi_\delta(\pi_0(t)y)|\xi_\delta(\pi_-(t)y  + \pi_+(t)y) - \xi_\delta(\pi_-(t)z + \pi_+(t)z)| \\
				& + \xi_\delta(\pi_-(t)z + \pi_+(t)z)|\xi_\delta(\pi_0(t)y - \xi_\delta(\pi_0(t)z)| \\
				&\leq \frac{2C}{\delta}\|y - z\|.
			\end{align*}
			Now, note that $\|y\| \leq \|\pi_0(t)y\| + \|(\pi_-(t) + \pi_+(t))y\|$ for all $y \in \mathbb{R}^n$. If $\|y\| \geq 4N\delta$, then $\max\{\|\pi_0(t)y\|, \|(\pi_-(t) + \pi_+(t))y\|\} \geq 2N\delta$, so that $\Xi_{\delta,t}(y) = 0$. Let $\delta_1 > 0$ be such as in \Cref{lem:local lipschitz} and fix $\delta > 0$ such that $4N\delta \leq \delta_1$. For $y, z \in \mathbb{R}^n$:
			\begin{align*}
				\|R_\delta(t,y) - R_\delta(t,z)\| &= \|[\Xi_{\delta,t}(y)R(t,y) - \Xi_{\delta,t}(y)R(t,z)] - [\Xi_{\delta,t}(z)R(t,z) - \Xi_{\delta,t}(y)R(t,z)]\| \\
				&\leq \Xi_{\delta,t}(y)\|R(t,y) - R(t,z)\| + |\Xi_{\delta,t}(y) - \Xi_{\delta,t}(z)|\|R(t,z)\|\\
				&\leq\begin{dcases}
					l_\delta(4N\delta)\|y - z\| + 8CNl_\delta(4N\delta)\|y - z\|, & \|y\|, \|z\| < 4N\delta, \\
					0, & \|y\|, \|z\| \geq 4N\delta,\\
					8CNl_\delta(4N\delta)\|y - z\|, & \|y\| \geq 4N\delta, \|z\| < 4N\delta,\\
				\end{dcases}\\
				&\leq l_\delta(4N\delta)(1 + 8CN)\|y - z\|,
			\end{align*}
			Hence, $L_\delta := l_\delta(4N\delta)(1 + 8CN)$ is a Lipschitz constant for $R_\delta(t,\cdot)$ for all $t \in \mathbb{R}$.
		\end{proof}
		
		\begin{corollary} \label{cor:Rdelta}
			For $\delta > 0$ sufficiently small, $\|R_\delta(t, y)\| \leq 4NL_\delta\delta$ for all $(t,y) \in \mathbb{R} \times \mathbb{R}^n$.
		\end{corollary}
		\begin{proof}
			Note that $R_\delta(t, 0) = 0$. This means that $\|R_\delta(t, y)\| \leq L_\delta\|y\|$. Obviously the claim holds if $\|y\| \leq 4N\delta$. On the other hand, if $\|y\| > 4N\delta$, then $R_\delta(t, y) = 0$ and so the proof is complete.
		\end{proof}
		
		Let us introduce for any $\eta \in (0,\min\{-a,b\}), s \in \mathbb{R}$ and a given $\delta$-modification of $R$, the \emph{substitution operator} $\tilde{R}_{\delta} : \BC_s^{\eta} \to \BC_s^{\eta}$ as
		\begin{equation*} 
			\tilde{R}_{\delta}(u) := R_\delta(\cdot,u(\cdot)),
		\end{equation*}
		and we show that this operator inherits the same properties as $R_\delta$.
		
		\begin{lemma} \label{lemma:substitution}
			Let $\eta \in (0, \min \{-a,b \}), s \in \mathbb{R}$ and $\delta > 0$ be sufficiently small. Then the substitution operator $\tilde{R}_\delta$ is well-defined and inherits the Lipschitz properties of $R_\delta$.
		\end{lemma}
		\begin{proof}
			It follows from \Cref{prop:global lipschitz} that 
			\begin{equation*}
				\|\tilde{R}_{\delta}(u)(t)\| = \|R_\delta(t,u(t))\| \leq L_\delta \|u(t)\|,
			\end{equation*}
			for all $u \in \BC_s^\eta$. Hence, $\|\tilde{R}_{\delta}(u) \|_{\eta,s} \leq L_\delta \| u \|_{\eta,s} < \infty$, i.e. $\tilde{R}_{\delta}(u) \in \BC_{s}^\eta$. The Lipschitz property follows immediately from \Cref{prop:global lipschitz} since
			\begin{equation*}
				\|\tilde{R}_{\delta}(u) - \tilde{R}_{\delta}(v)\|_{\eta,s} \leq  L_\delta \|u - v\|_{\eta,s},
			\end{equation*}
			for all $u,v \in \BC_s^\eta$, and so $\|\tilde{R}_{\delta}(u)\|_{\eta,s} \leq 4NL_\delta\delta$.
		\end{proof}
		Define for any $\eta \in (0, \min \{-a,b \})$ and $s \in \mathbb{R}$ the linear operator $U_s^\eta: E_0(s) \rightarrow \BC_s^\eta$ by 
		\begin{equation*}
			(U_s^\eta y_0)(t) := U(t, s)y_0.
		\end{equation*}
		\begin{lemma} \label{lemma:Uetas}
			Let $\eta \in (0, \min \{-a,b \})$ and $s \in \mathbb{R}$. Then the operator $U_s^\eta$ is well-defined and bounded.
		\end{lemma}
		\begin{proof}
			Let $\varepsilon \in (0,\eta]$ be given. It follows from \Cref{prop:criteria} that
			\begin{equation*}
				\|U_s^\eta y_0\|_{\eta,s} \leq C_\varepsilon\|y_0\|\sup_{t \in \mathbb{R}}e^{(\varepsilon - \eta)|t - s|} = C_\varepsilon\|y_0\|,  
			\end{equation*}
			for all $y_0 \in E_0(s)$, and so $U_s^\eta$ is well-defined and bounded.
		\end{proof}

		\subsection{Existence of a Lipschitz center manifold}
		\label{subsec: Lipschitz CM}
		Our next goal is to define a parameterized fixed point operator such that its fixed points correspond to (exponentially) bounded solutions on $\mathbb{R}$ of the modified equation
		\begin{equation} \label{eq:variation constants Rdelta}
			u(t) = U(t,s)u(s) + \int_s^t U(t,\tau)R_\delta(\tau, u(\tau)) d\tau,
		\end{equation}
		for all $(t,s) \in \mathbb{R}^2$ and some small $\delta > 0$. For a given $\eta \in (0,\min \{-a,b \}), s\in \mathbb{R}$ and sufficiently small $\delta > 0$, we define the fixed point operator $ \mathcal{G}_s^\eta : \BC_s^\eta \times E_0(s) \to \BC_s^\eta$ as
		\begin{equation*}
			\mathcal{G}_s^\eta(u,y_0) := U_s^\eta y_0 + \mathcal{K}_s^\eta(\tilde{R}_{\delta}(u)).
		\end{equation*}
		It follows from \Cref{prop:ketas}, \Cref{lemma:substitution} and \Cref{lemma:Uetas} that $\mathcal{G}_s^\eta$ is well-defined. We first show that $\mathcal{G}_s^\eta(\cdot,y_0)$ admits a unique fixed point and is globally Lipschitz for all $y_0 \in E_0(s)$.
		
		\begin{proposition} \label{thm:contraction}
			Let $\eta \in (0, \min\{-a,b\})$ and $s \in \mathbb{R}$. If $\delta > 0$ is sufficiently small, then the following statements hold.
			\begin{enumerate}
				\item[$1.$] For every $y_0 \in E_0(s)$, the map $\mathcal{G}_s^\eta(\cdot,y_0)$ has a unique fixed point $\hat{u}_s^\eta(y_0)$.
				\item[$2.$] The map $\hat{u}_s^\eta: E_0(s) \rightarrow \BC_s^\eta$ is globally Lipschitz and $\hat{u}_s^\eta(0) = 0$.
			\end{enumerate}
		\end{proposition}
		
		\begin{proof}
			Fix $\varepsilon \in (0,\eta]$. For $u,v \in \BC_s^\eta$ and $y_0, z_0 \in E_0(s)$, we have
			\begin{align*}
				\|\mathcal{G}_s^\eta(u,y_0) - \mathcal{G}_s^\eta(v,z_0)\|_{\eta,s} 
				&\leq \sup_{t \in \mathbb{R}}e^{-\eta|t - s|}\|U_0(t,s)(y_0 - z_0)\|
				+ L_\delta\|\mathcal{K}_s^\eta\|\|u - v\|_{\eta,s} \\
				&\leq C_\varepsilon\|y_0 - z_0\| + L_\delta\|\mathcal{K}_s^\eta\|\|u - v\|_{\eta,s}.
			\end{align*}
			To prove the first assertion, set $y_0 = z_0$ and choose $\delta > 0$ small enough such that $L_\delta\|\mathcal{K}_s^\eta\| \leq \frac{1}{2}$ (\Cref{prop:global lipschitz}) since then
			\begin{equation*}
				\|\mathcal{G}_s^\eta(u,y_0) - \mathcal{G}_s^\eta(v,y_0)\|_{\eta,s} \leq \frac{1}{2}\|u - v\|_{\eta,s}.
			\end{equation*}
			Since $\BC_s^\eta$ is a Banach space, the contracting mapping principle applies and so the contraction $\mathcal{G}_s^\eta(\cdot,y_0)$ has a unique fixed point, say $\hat{u}_s^\eta(y_0)$.
			
			To prove the second assertion, let $\hat{u}_s^\eta(y_0)$ and $\hat{u}_s^\eta(z_0)$ be the unique fixed points of $\mathcal{G}_s^\eta(\cdot,y_0)$ and $\mathcal{G}_s^\eta(\cdot,z_0)$ respectively. Then,
			\begin{align*}
				\|\hat{u}_s^\eta(y_0) - \hat{u}_s^\eta(z_0)\|_{\eta,s} 
				= \|\mathcal{G}_s^\eta(\hat{u}_s^\eta(y_0),y_0) - \mathcal{G}_s^\eta(\hat{u}_s^\eta(z_0),z_0)\|_{\eta,s} \leq C_\varepsilon \|y_0 - z_0\| + \frac{1}{2}\|\hat{u}_s^\eta(y_0) - \hat{u}_s^\eta(z_0)\|_{\eta,s}.
			\end{align*}
			This implies that $\|\hat{u}_s^\eta(y_0) - \hat{u}_s^\eta(z_0)\|_{\eta,s} \leq 2C_\varepsilon \|y_0 - z_0\|$, and so $\hat{u}_s^\eta$ is globally Lipschitz. Since $\hat{u}_s^\eta(0) = \mathcal{G}_s^\eta(\hat{u}_s^\eta(0),0) = 0,$ the second assertion follows.
		\end{proof}
		
		In order to construct a center manifold, define the \emph{center bundle} $E_0 := \{(s,y_0) \in \mathbb{R} \times \mathbb{R}^n : y_0 \in E_0(s)\}$ and the map $\mathcal{C}: E_0 \rightarrow \mathbb{R}^n$ by
		\begin{equation} \label{eq:mapC}
			\mathcal{C}(s,y_0) := \hat{u}_s^\eta(y_0)(s).
		\end{equation}
		
		\begin{definition} \label{def:Wc}
			A \emph{global center manifold} of \eqref{eq:variation constants Rdelta} is defined as the image $\mathcal{W}^c := \mathcal{C}(E_0)$, whose $s$-\emph{fibers} are defined as $\mathcal{W}^c(s) := \{\mathcal{C}(s, y_0) \in \mathbb{R}^n : y_0 \in E_0(s)\}.$
		\end{definition} 
		
		Recall from \Cref{thm:contraction} that for a fixed $s \in \mathbb{R}$, the map $\hat{u}_s^\eta$ is globally Lipschitz. Hence, the map $\mathcal{C}(s,\cdot) : E_0(s) \to \mathbb{R}^n$ is globally Lipschitz, where the Lipschitz constant depends on $s$, i.e. $\mathcal{C}$ is only \emph{fiberwise Lipschitz}. The following result shows that the Lipschitz constant can be chosen independently of the fiber, and so we can say that $\mathcal{W}^c$ is a Lipschitz global center manifold of \eqref{eq:variation constants Rdelta}.
		
		\begin{lemma} \label{lemma:lipschitzCMT}
			There exists a constant $L > 0$ such that $\|\mathcal{C}(s,y_0) - \mathcal{C}(s,z_0)\| \leq L \|y_0 - z_0\|$ for all $(s,y_0),(s,z_0) \in E_0$. 
		\end{lemma}
		\begin{proof}
			Let $(s, y_0), (s, z_0) \in E_0$ be given. It follows from \Cref{lemma:substitution} and \Cref{thm:contraction} that
			\begin{align*}
				\|\mathcal{C}(s,y_0) - \mathcal{C}(s,z_0)\|
				&= \|\mathcal{G}_s^\eta(\hat{u}_s^\eta(y_0),y_0)(s) - \mathcal{G}_s^\eta(\hat{u}_s^\eta(z_0),z_0)(s)\| \\
				&\leq \|y_0 - z_0\| + \|\mathcal{K}_{s}^\eta\|\|\tilde{R}_{\delta}(\hat{u}_s^\eta(y_0)) - \tilde{R}_{\delta}(\hat{u}_s^\eta(z_0))\|_{\eta,s} \\
				&\leq \|y_0 - z_0\| + L_\delta\|\mathcal{K}_{s}^\eta\|\|\hat{u}_s^\eta(y_0) - \hat{u}_s^\eta(z_0)\|_{\eta,s} \leq (1 + 2C_\varepsilon L_\delta \|\mathcal{K}_{s}^\eta\|)\|y_0 - z_0\|.
			\end{align*}
			Hence, $L := 1 + 2C_\varepsilon L_\delta \|\mathcal{K}_{s}^\eta\|$ is a Lipschitz constant that is independent of $s$ by \Cref{prop:ketas}.
		\end{proof}
		Recall from the definition of the $\delta$-modification of $R$ that $R_\delta = R$ on $\mathbb{R} \times B(0,\delta)$. Hence, the modified integral equation \eqref{eq:variation constants Rdelta} is equivalent to the original integral equation \eqref{eq:integraleq}, and by \Cref{lemma:integralequivalence} to the ordinary differential equation \eqref{eq:A(t) R}, on $B(0,\delta)$.
		\begin{definition} \label{def:Wcloc}
			A \emph{local center manifold} of \eqref{eq:A(t) R} is defined as the image
			\begin{equation*}
				\mathcal{W}^c_{\loc} := \mathcal{C}(\{(s,y_0) \in E_0 : \mathcal{C}(s,y_0) \in B(0,\delta)\}).
			\end{equation*}
		\end{definition}
		In the definitions of the center manifolds and their associated fiber bundles (\Cref{def:Wc} and \Cref{def:Wcloc}), we used the map $\mathcal{C}$ to explicitly construct these objects. However, sometimes one likes to think of the center manifold as the graph of a function. To obtain such a representation, define the map $\mathcal{H} : E_0 \to \mathbb{R}^n$ as $\mathcal{H}(s,y_0) := (I-\pi_0(s))\mathcal{C}(s,y_0)$ and notice from \Cref{prop:ketas} that we have the decomposition $\mathcal{C}(s,y_0) = y_0 + \mathcal{H}(s,y_0)$ in the nonhyperbolic and hyperbolic part respectively. Hence, we can write for example
		\begin{align} \label{eq:graphH}
			\mathcal{W}_c(s) = \{y_0 + \mathcal{H}(s,y_0) : y_0 \in E_0(s) \} 
			\cong \{(y_0,\mathcal{H}(s,y_0)) : y_0 \in E_0(s) \}
			=\mbox{graph}(\mathcal{H}(s,\cdot)),
		\end{align}
		and since $E_0(s)$ and $E_+(s) \oplus E_{-}(s)$ have only zero in their intersection (\Cref{prop:criteria}), this identification makes sense. Notice that the map $\mathcal{H}$, identified as a graph in \eqref{eq:graphH}, is strictly speaking a map that takes values in $E_{+}(s) \oplus E_{-}(s)$. Similar graph-like representations can be obtained for $\mathcal{W}^c$ and $\mathcal{W}_{\loc}^c$.
		
		\section{Properties of the center manifold} \label{sec:properties}
		In this section, we prove that $\mathcal{W}_{\loc}^c$ is locally invariant and consists of slow dynamics. Moreover, we prove that the center manifold inherits the same finite order of smoothness as the nonlinearity $R$ and its tangent bundle is precisely the center bundle $E_0$. Lastly, we prove that the center manifold is $T$-periodic in a neighborhood of the origin. At the end of the section, we combine all the results to prove \Cref{thm:main}.
		
		Our first aim is to prove the local invariance property of $\mathcal{W}_{\loc}^c$. Therefore, let $S_\delta(t,s,\cdot) : \mathbb{R}^n \to \mathbb{R}^n$ denote the (time-dependent) flow of
		\begin{equation} \label{eq:yRdelta}
			\dot{y}(t) = A(t)y(t) + R_\delta(t,y(t)).
		\end{equation}
		Moreover, \Cref{lemma:integralequivalence} still holds when $R$ is replaced by $R_\delta$ and so the ordinary differential equation \eqref{eq:yRdelta} is equivalent to the integral equation \eqref{eq:variation constants Rdelta}. By (local) uniqueness of solutions, we have that \eqref{eq:Sts} still holds with $S$ replaced by $S_\delta$. The following result is the nonlinear analogue of \Cref{prop:X0s} and is a preliminary result to prove in \Cref{prop:invariance} the local invariance property of the center manifold.
		
		\begin{proposition} \label{prop:W^c(s)}
			Let $\eta \in (0, \min\{-a,b\})$ and $s \in \mathbb{R}$. Then
			\begin{align*}
				\mathcal{W}^c(s) = \{y_0 \in \mathbb{R}^n : \text{ there exists a solution of \eqref{eq:yRdelta}} 
				\text{ through $y_0$ belonging to $\BC_s^\eta$}\}.
			\end{align*}
		\end{proposition}
		\begin{proof}
			Choose $y_0 \in \mathcal{W}^c(s)$, then $y_0 = \mathcal{C}(s, z_0) = \hat{u}_s^\eta(z_0)(s)$ for some $z_0 \in E_0(s)$. \Cref{prop:ketas} shows that $\mathcal{K}_s^\eta \tilde{R}_{\delta}(u)$ is the unique solution of \eqref{eq:inhomogeneous CMT} with $f = \tilde{R}_{\delta}(u)$. Since $u = \hat{u}_s^\eta(z_0)$ is a fixed point of $\mathcal{G}_s^\eta(\cdot,z_0)$, we get
			\begin{align*}
				u(t) 
				&= U(t,s)z_0 + (\mathcal{K}_s^\eta \tilde{R}_{\delta}(u))(t) \\
				&= U(t,s)z_0  + U(t,s)(\mathcal{K}_s^\eta \tilde{R}_{\delta}(u))(s)
				+ \int_s^tU(t, \tau){R}_{\delta}(\tau, u(\tau))d\tau \\
				&= U(t,s)u(s) + \int_s^tU(t, \tau){R}_{\delta}(\tau, u(\tau))d\tau.
			\end{align*}
			for all $(t,s) \in \mathbb{R}^2$. Hence, $u = \hat{u}_s^\eta(z_0)$ is a solution of \eqref{eq:variation constants Rdelta}, and so \eqref{eq:yRdelta}, through $u(s) = y_0$ which belongs to $\BC_s^\eta$. Conversely, let $y_0 \in \mathbb{R}^n$ such that there exists a solution $u$ of \eqref{eq:yRdelta}, and so \eqref{eq:variation constants Rdelta}, in $\BC_s^\eta$ satisfying $u(s) = y_0$. It follows from \Cref{prop:ketas} that
			\begin{equation*}
				u(t) = U(t, s)\pi_0(s)u(s) + (\mathcal{K}_s^\eta \tilde{R}_{\delta}(u))(t).
			\end{equation*}
			Hence, $u = \mathcal{G}_s^\eta (u, \pi_0(s)u(s))$ so $y_0 = u(s) = \mathcal{C}(s, \pi_0(s)y_0) \in \mathcal{W}^c(s)$ by uniqueness of the fixed point.
		\end{proof}

		\begin{proposition} \label{prop:invariance}
			The local center manifold $\mathcal{W}_{\loc}^c$ has the following properties.
			\begin{enumerate}
				\item[$1.$] $\mathcal{W}^c_{\loc}$ is \emph{locally invariant}: if $(s, y_0) \in \mathbb{R} \times \mathcal{W}^c_{\loc}$ and $t_{-},t_{+} \in \mathbb{R}$ with $ s \in (t_{-}, t_+)$ such that $S(t, s, y_0) \in B(0, \delta)$ for all $t \in (t_{-},t_{+})$, then $S(t, s, y_0) \in \mathcal{W}^c_{\loc}$.
				\item[$2.$] $\mathcal{W}_{\loc}^c$ contains every solution of \eqref{eq:A(t) R} that exists on $\mathbb{R}$ and remains sufficiently small for all positive and negative time: if $u: \mathbb{R} \rightarrow B(0,\delta)$ is a solution of \eqref{eq:A(t) R}, then $u(t) \in \mathcal{W}^c_{\loc}$ for all $t \in \mathbb{R}$.
				\item[$3.$] If $(s, y_0) \in \mathbb{R} \times \mathcal{W}^c_{\loc}$, then $S(t, s, y_0) = \hat{u}_t^\eta(\pi_0(t)S(t, s, y_0))(t) = \mathcal{C}(t, \pi_0(t)S(t, s, y_0))$ for all $t \in (t_{-},t_{+})$.
				\item[$4.$] $0 \in \mathcal{W}^c_{\loc}$ and $\mathcal{C}(t, 0) = 0$ for all $t \in \mathbb{R}$.
			\end{enumerate}
		\end{proposition}
		\begin{proof}
			We prove the four assertions step by step.
			
			1. By \Cref{prop:W^c(s)}, choose a solution $u \in \BC_s^\eta$ of \eqref{eq:yRdelta} such that $u(s) = y_0$. Note that $u(s) = S_\delta(s, s, y_0)$, so by uniqueness $u(t) = S_\delta(t, s, y_0)$ for all $t \in (t_{-},t_{+})$. Then $S_\delta(t, s, y_0) \in \mathcal{W}^c(t) \subset \mathcal{W}^c$. Since $S_\delta(t, s, y_0) \in B(0, \delta)$, it follows that $S(t,s,y_0) = S_\delta(t, s, y_0) \in \mathcal{W}^c_{\loc}$.
			
			2. Recall that \eqref{eq:A(t) R} and \eqref{eq:yRdelta} are equal on $B(0, \delta)$. If $u$ is such a solution, then $u \in \BC_{s}^\eta$. The assumption that $u$ takes values in $B(0,\delta)$ and \Cref{prop:W^c(s)} together imply with the first assertion the result. 
			
			3. In the proof of \Cref{prop:W^c(s)} it is shown that $y_0 = \mathcal{C}(s, \pi_0(s)y_0)$ for any $y_0 \in \mathcal{W}^c(s)$. So it is certainly true for $y_0 \in \mathcal{W}^c_{\loc}$ that $S_\delta(s, s, y_0) = y_0 = \hat{u}_s^\eta(\pi_0(s)S_\delta(s, s, y_0))(s) = \mathcal{C}(s, \pi_0(s)S_\delta(s, s, y_0))$. Because $\mathcal{W}^c_{\loc}$ is locally invariant, we have that $S_\delta(t, s, y_0) \in \mathcal{W}^c_{\loc}$ for all $t \in \mathbb{R}$ sufficiently close to $s$ and by uniqueness of solutions, $S_\delta(t, s, y_0) = \hat{u}_t^\eta(\pi_0(t)S_\delta(t, s, y_0)) = \mathcal{C}(t, \pi_0(t)S_\delta(t, s, y_0))$. Since we are on the local center manifold, we can replace $S_\delta$ with $S$.
			
			4. Notice that $\mathcal{C}(t,0) = \hat{u}_t^\eta(0)(t) = 0$ for all $t \in \mathbb{R}$, where the last equality follows from \Cref{thm:contraction}. Clearly, $0 = \mathcal{C}(t,0) \in \mathcal{W}_{\loc}^{c}$ and so the proof is complete.
		\end{proof}
		
		It is now possible to explain the fact that the dynamics on the center manifold is rather slow. Indeed, the local invariance of $\mathcal{W}_{\loc}^{c}$ (\Cref{prop:invariance}) in combination with \Cref{prop:W^c(s)} shows that solutions on the center manifold are in $\BC_s^\eta$ for some sufficiently small $\eta > 0$, i.e. their asymptotic behavior forward and backward in time can only be a limited exponential. 
		
		The next step is to show that the map $\mathcal{C}$ inherits the same order of smoothness as the time-dependent nonlinear perturbation $R$. Proving additional smoothness of center manifolds requires work. A well-known technique to increase smoothness of center manifolds is via the theory of contraction on scales of Banach spaces \cite{Vanderbauwhede1987}. Since this part of the theory is rather technical, it is delegated to Appendix A. The main result is presented in \Cref{thm:smoothnessC} and simply states that the map $\mathcal{C}$ is $C^k$-smooth and so $\mathcal{W}^c$ and $\mathcal{W}_{\loc}^c$ are both $C^k$-smooth manifolds in $\mathbb{R}^n$. The additional regularity of the center manifold allows us to study its tangent bundle.
		
		\begin{proposition} \label{prop:bundle}
			The tangent bundle of $\mathcal{W}^c$ and $\mathcal{W}_{\loc}^c$ is $E_0$: $D_2\mathcal{C}(s, 0)y_0 = y_0$ for all $(s,y_0) \in E_0$.
		\end{proposition}
		\begin{proof}
			Let $ \eta \in [\eta_{-},\eta_{+}] \subset (0,\min\{-a,b\})$ such that $k \eta_{-} < \eta_{+}$. Differentiating
			\begin{equation*}
				\hat{u}_s^{\eta_{-}}(y_0) = U_s^{\eta_{-}} y_0 + \mathcal{K}_s^{{\eta_{-}}} \circ \tilde{R}_{\delta}(\hat{u}_s^{\eta_{-}}(y_0))
			\end{equation*}
			with respect to $y_0$ yields
			\begin{equation*}
				D\hat{u}_s^{\eta_{-}}(y_0) = U_s^{\eta_{-}} + \mathcal{K}_s^{{\eta_{-}}} \circ \tilde{R}_{\delta}^{{(1)}}(\hat{u}_s^{\eta_{-}}(y_0)) \circ D\hat{u}_s^{\eta_{-}}(y_0).
			\end{equation*}
			Setting $y_0 = 0$ and recalling the fact that $\hat{u}_s^{\eta_{-}}(0) = \tilde{R}_{\delta}^{(1)}(0) = 0$ shows that $D\hat{u}_s^{\eta_{-}}(0) = U_s^{\eta_{-}}$. If $\eva_s : \BC_s^\eta \to \mathbb{R}^n : f \mapsto f(s)$ denotes the bounded linear evolution operator (at time $s$), then
			\begin{equation*}
				D_2\mathcal{C}(s,0) = \eva_s(D(\mathcal{J}_s^{\eta, \eta_{-}} \circ \hat{u}_{s}^{\eta_{-}})(0)) = \eva_s(U_s^{\eta}) = I,
			\end{equation*}
			which proves the claim.    
		\end{proof}
		Since our original system \eqref{eq:A(t) R} is $T$-periodic, it is not surprising that the center manifold itself is $T$-periodic in a neighborhood of zero. To prove this, let us define for all $s \in \mathbb{R}$ and sufficiently small $\delta > 0$ the map $N_s : E_0(s) \to E_0(s)$ by
		\begin{equation*}
			N_s(y_0) := \pi_0(s)S_\delta(s + T,s,\mathcal{C}(s,y_0)).
		\end{equation*}
		\begin{lemma} \label{lemma:invertibility}
			The function $N_s$ is invertible in a neighborhood of the origin. Moreover, this neighborhood can be written as $U \cap E_0(s)$ for some open neighborhood $U \subset \mathbb{R}^n$ of zero, independent of $s$.
		\end{lemma}
		\begin{proof}
			Recall the well-known standard result: $D[S_\delta(t, s, \cdot)](0) = U(t,s)$ for all $(t,s) \in \mathbb{R}^2$, see for instance \cite[Section 17.6]{Hirsch2013}. The differential of $N_s$ in $0$ is given by
			\begin{align*}
				DN_s(0)
				&= \pi_0(s) \circ D[S_\delta(s + T, s, \cdot)](\mathcal{C}(s, 0)) \circ D[\mathcal{C}(s,\cdot)](0) \\
				&= \pi_0(s) \circ D[S_\delta(s + T, s, \cdot)](0)\\
				&= \pi_0(s)U_0(s + T, s)  = U_0(s + T, s),
			\end{align*}
			where we used \Cref{prop:criteria}, \Cref{prop:bundle}, and the fact that $U_0(s + T, s)y_0 \in E_0(s + T) = E_0(s)$. It follows from \Cref{prop:criteria} that $DN_s(0) = U_0(s+T,s)$ is a bounded linear isomorphism and so $N_s$ is locally invertible by the inverse function theorem. To prove that the neighborhood may be written as claimed, let us first observe that for a given $\varepsilon > 0$ there holds
			\begin{align*}
				\| DN_s(y_0) - DN_s(0) \|
				&\leq \|U_0(s+T,s)\pi_{0}(s) \| \| D\mathcal{C}(s,y_0) - D\mathcal{C}(s,0) \| \\
				& \leq N C_\varepsilon e^{\varepsilon T} \| D\mathcal{C}(s,y_0) - D\mathcal{C}(s,0) \| \\
				& \leq N C_\varepsilon e^{\varepsilon T} L(1) ||y_0|| \to 0, \quad \mbox{ as } y_0 \to 0,
			\end{align*}
			due to \Cref{prop:criteria} and \Cref{cor:LipschitzC}. Hence, $DN_s(y_0)$ is uniformly convergent (in the variable $s$) as $y_0 \to 0$ and so the implicit function may be defined on a neighborhood that does not depend on $s$.
		\end{proof}
		
		\begin{proposition}
			There exists a $\delta > 0$ such that $\mathcal{C}(s + T, y_0) = \mathcal{C}(s, y_0)$ for all $(s,y_0) \in E_0$ satisfying $\|y_0\| \leq \delta$.    
		\end{proposition}
		\begin{proof}
			Let $(s,y_0) \in E_0$ be given. By \Cref{lemma:invertibility}, choose $\delta > 0$ such that if $\|y_0\| \leq \delta$, it is possible to write $y_0 = N_s(z_0)$. It follows from \Cref{prop:criteria}, \Cref{prop:invariance} and \eqref{eq:Sts} that
			\begin{align*}
				\mathcal{C}(s + T, y_0)
				&= \mathcal{C}(s + T, \pi_0(s)S_\delta(s + T, s, \mathcal{C}(s, z_0))) \\
				&= S_\delta(s + T, s, \mathcal{C}(s, z_0)) \\
				&= S_\delta(s, s - T, \mathcal{C}(s, z_0)) \\
				&= \mathcal{C}(s, \pi_0(s)S_\delta(s, s - T, \mathcal{C}(s, z_0))) \\
				&= \mathcal{C}(s, \pi_0(s)S_\delta(s + T, s, \mathcal{C}(s, z_0))) = \mathcal{C}(s, y_0).
			\end{align*}
			This proves the $T$-periodicity of the center manifold.  
		\end{proof}
		
		Recall that \eqref{eq:A(t) R} was just a time-dependent translation of \eqref{eq:ODE} via the given periodic solution. Hence, if $x$ is a solution of \eqref{eq:ODE} then $y = x - \gamma$ is a solution of \eqref{eq:A(t) R} and so
		\begin{align}
			\mathcal{W}_{\loc}^{c}(\Gamma) := \{ \gamma(s) + \mathcal{C}(s,y_0) \in \mathbb{R}^n: 
			(s,y_0) \in E_0 \mbox{ and } \mathcal{C}(s,y_0) \in B(0,\delta) \}
		\end{align}
		is a $T$-periodic $C^k$-smooth $(n_0+1)$-dimensional manifold in $\mathbb{R}^n$ defined in the vicinity of $\Gamma$ for a sufficiently small $\delta > 0$. To see this, recall that $\gamma$ is $T$-periodic and $C^k$-smooth together with the fact that $\mathcal{C}$ is $T$-periodic in the first component and $C^k$-smooth. Recall from \Cref{prop:invariance} that $\mathcal{C}(t,0) = 0$ and so $\Gamma \subset \mathcal{W}_{\loc}^{c}(\Gamma)$. We call $\mathcal{W}_{\loc}^{c}(\Gamma)$ a \emph{local center manifold around $\Gamma$} and notice that this manifold inherits all the properties of $\mathcal{W}_{\loc}^c$, which proves \Cref{thm:main}.
		
		\section{Examples and counterexamples} \label{sec:examples}
		It is widely known that center manifolds for equilibria have interesting qualitative properties \cite{Guckenheimer1983,Kuznetsov2023a,Sijbrand1985}. For example, such center manifolds are not necessarily unique and are not necessarily of the class $C^{\infty}$ even if the vector field is $C^{\infty}$-smooth. Of course, there always exists for $C^{\infty}$-smooth systems an open neighborhood $U_k$ around the equilibrium such that a center manifold is $C^k$-smooth on $U_k$. However, the neighborhood $U_k$ may shrink towards the equilibrium as $k \to \infty$, see \cite{Strien1979,Takens1984,Sijbrand1985} for explicit examples. When the vector field is analytic, there is the possibility of the existence of a non-analytic $C^{\infty}$-smooth center manifold, see \cite{Kelley1967,Sijbrand1985} for explicit examples. It is studied in \cite{Sijbrand1985} under which conditions a unique, $C^{\infty}$-smooth or analytic center manifold exists. 
		
		The aim of this section is to provide several explicit examples illustrating similar behavior for the periodic center manifolds. For example, we provide in \Cref{ex:analytic} an analytic $2\pi$-periodic two-dimensional center manifold near a nonhyperbolic cycle that is a cylinder. To complete the periodic two-dimensional center manifold theory, we provide in \Cref{ex:mobius} a system that admits a $2\pi$-periodic two-dimensional center manifold near a nonhyperbolic cycle that is a M\"obius band. Both examples are minimal polynomial vector fields admitting a cylinder or M\"obius band as a periodic center manifold.
		
		To illustrate the existence of a non-unique and non-analytic $C^{\infty}$-smooth periodic center manifold, we will study the analytic nonlinear periodically driven system
		\begin{equation} \label{eq:Ex2}
			\begin{dcases}
				\dot{x} = -x^2, \\
				\dot{y} = -y + \sin(t)x^2.
			\end{dcases}
		\end{equation}
		This vector field is a modification of the vector field used in \cite{Kelley1967} to illustrate the non-unique behavior of center manifolds for equilibria. Note that the system \eqref{eq:Ex2} is not of the form \eqref{eq:ODE} but already written in the style of \eqref{eq:A(t) R} where $A$ is autonomous and $R$ is $2\pi$-periodic in the first argument. The assumption of an autonomous linear part is not a restriction. Indeed, the Floquet normal form $U(t,0) = Q(t)e^{Bt}$, where $Q$ is $T$-periodic, $Q(0) = I$, $Q(t)$ is invertible for all $t \in \mathbb{R}$, and the matrix $B \in \mathbb{C}^{n \times n}$ that satisfies $U(T,0) = e^{BT}$ shows that a general system of the form \eqref{eq:A(t) R} is equivalent to the nonlinear periodically driven system
		\begin{equation} \label{eq:transformedA(t)R}
			\dot{z}(t) = Bz(t) + G(t,z(t)),
		\end{equation}
		where $G(t,z(t)) := Q(t)^{-1}R(t,Q(t)z(t))$, and $z$ satisfies the Lyapunov-Floquet transformation $y(t) = Q(t)z(t)$. Clearly \eqref{eq:transformedA(t)R} has an autonomous linear part and $G$ is $T$-periodic in the first component since $R$ and $Q$ are both $T$-periodic. Notice that the whole periodic center manifold construction from previous subsections still applies for systems of the form \eqref{eq:transformedA(t)R}. The reason we study systems of the form \eqref{eq:transformedA(t)R} instead of a general system of the form \eqref{eq:ODE} is to keep the calculations rather simple. Indeed, if one would like to cook up an explicit non-trivial example of a periodic center manifold near a nonhyperbolic cycle, one needs to be able to compute explicitly the periodic solution, the fundamental matrix and its associated Floquet multipliers, which is rather difficult for general systems of the form \eqref{eq:ODE}. We remark that the computations for the periodic center manifolds of simple periodically driven systems considered in this section are rather tedious compared with their equilibrium analogues.
		
		In \Cref{ex:nonanalytic} we show that \eqref{eq:Ex2} admits a non-analytic $2\pi$-periodic center manifold. Next, we show in \Cref{ex:nonunique} that \eqref{eq:Ex2} admits a $2\pi$-periodic center manifold that is  \emph{locally (non)-unique}, i.e. there exist subneighborhoods of any neighborhood of $\mathbb{R} \times \{0\}$ where the center manifold is unique and others where it is not unique. This freedom of different center manifolds allows us to choose in \Cref{ex:Cinftysmooth} a particular $2\pi$-periodic center manifold for \eqref{eq:Ex2} that is $C^{\infty}$-smooth. Hence, we have shown that analytic vector fields can admit non-analytic $C^{\infty}$-smooth periodic center manifolds. To complete this list of examples, we will show in \Cref{ex:nonCinfitysmooth} that the $C^{\infty}$-smooth (analytic) nonlinear periodically driven system
		\begin{equation} \label{eq:Exshrinking}
			\begin{dcases}
				\dot{x} = xz - x^3, \\
				\dot{y} = y + (1+\sin(t))x^2, \\
				\dot{z} = 0.
			\end{dcases}
		\end{equation}
		admits a $2\pi$-periodic non-$C^\infty$-smooth center manifold. This vector field is a modification of the vector field used in \cite{Strien1979} to illustrate the non-$C^\infty$-smoothness of center manifolds near equilibria. Hence, we have proven that there exists analytic vector fields admitting locally (non)-unique, (non)-$C^{\infty}$-smooth and (non)-analytic periodic center manifolds.

		\begin{example} \label{ex:analytic}
			The analytic system
			\begin{equation} \label{eq:Ex1}
				\begin{dcases}
					\dot{x}_1 = x_1 - x_2 - x_1(x_1^2 + x_2^2), \\
					\dot{x}_2 = x_1 + x_2 - x_2(x_1^2 + x_2^2), \\
					\dot{x}_3 = 0,
				\end{dcases}
			\end{equation}
			admits an analytic $2\pi$-periodic two-dimensional center manifold that is a cylinder.
		\end{example}
		\begin{proof}
			Notice that \eqref{eq:Ex1} admits a $2\pi$-periodic solution $\gamma(t) = (\cos(t),\sin(t),0)$. The system around $\Gamma = \gamma(\mathbb{R})$ can be written in coordinates  $x = \gamma + y$ as
			\begin{equation} \label{eq:Ex1ODE}
				\begin{dcases}
					\dot{y}_1 = -2\cos^2(t)y_1 -(1+2\sin(t)\cos(t))y_2 - 3\cos^2(t)y_1^2 -2 \sin(t)y_1y_2 - \cos(t)y_2^2 - y_1^3 -y_1y_2^2,\\
					\dot{y}_2 = (1-2\sin(t) \cos(t))y_1 -2 \sin^2(t)y_2 - \sin(t)y_1^2 -2\cos(t)y_1y_2 -3 \sin^2(t)y_2^2 -y_1^2y_2 - y_2^3, \\
					\dot{y}_3 = 0,   
				\end{dcases}
			\end{equation}   
			where $y:= (y_1,y_2,y_3)$ and $x := (x_1,x_2,x_3)$. The linearization around the origin of \eqref{eq:Ex1ODE} reads
			\begin{equation*}
				A(t) = 
				\begin{pmatrix}
					-2 \cos^2(t) & -2\sin(t)\cos(t)-1 & 0 \\
					-2 \sin(t)\cos(t) + 1 & -2\sin^2(t) & 0 \\
					0 & 0 & 0
				\end{pmatrix}
			\end{equation*}
			and so the solution of the variational equation around $\Gamma$ is generated by the fundamental matrix $U(t,s) = V(t)V(s)^{-1}$ where
			\begin{equation*}
				V(t) = 
				\begin{pmatrix}
					e^{-2t}\cos(t)  & - \sin(t) & 0 \\
					e^{-2t}\sin(t) & \cos(t)  & 0 \\
					0 & 0 & 1
				\end{pmatrix}.
			\end{equation*}
			The Floquet multipliers are given by $\lambda_1 = 1, \lambda_2 = e^{-4\pi}$ and $\lambda_3 = 1$. Hence, the center subspace and stable subspace (at time $t$) can be obtained as $E_0(t) = \spn \{\zeta_1(t),\zeta_3(t) \}$ and $E_{-}(t) = \spn \{\zeta_2(t) \}$ respectively, where $\zeta_1(t) = (-\sin(t),\cos(t),0), \zeta_2(t) = (\cos(t),\sin(t),0)$ and $\zeta_3(t) = (0,0,1)$. The center bundle $E_0$ parametrizes a cylinder as a ruled surface since
			\begin{equation*}
				(x_1(t,v),x_2(t,v),x_3(t,v)) = \gamma(t) + v \zeta_3(t),
			\end{equation*}
			for all $t,v \in \mathbb{R}$. It follows from \Cref{sec:properties} that for any $k \geq 1$ there exists a $2\pi$-periodic $C^k$-smooth two-dimensional locally invariant center manifold $\mathcal{W}_{\loc}^c$ for \eqref{eq:Ex1ODE} around the origin that is tangent to $E_0$. To obtain this center manifold, let us transform \eqref{eq:Ex1ODE} into eigenbasis, i.e. we perform the change of variables from $y$ to $z := (z_1,z_2,z_3)$ as
			\begin{align} \label{eq:Ex1variables}
				z_1 &= -\sin(t) y_1 + \cos(t)y_2, \nonumber\\
				z_2 &= \cos(t)y_1 + \sin(t)y_2, \\
				z_3 &= y_3 \nonumber,
			\end{align}
			to obtain the autonomous system
			\begin{equation} \label{eq:Ex1autono}
				\begin{dcases}
					\dot{z}_1 = -z_1(z_1^2 +z_2^2 + 2z_2), \\
					\dot{z}_2 = -(z_2+1)(z_1^2 +z_2^2 + 2z_2), \\
					\dot{z}_3 = 0.
				\end{dcases}
			\end{equation}
			The $z_1z_3$-plane corresponds to the center subspace while the $z_2$-axis corresponds to the stable subspace. Therefore, the center manifold is parametrized by $z_2(t) = \mathcal{H}(t,z_1,z_3)$, where $\mathcal{H}$ is $2\pi$-periodic in the first variable and consists solely of nonlinear terms in the last two variables. Because \eqref{eq:Ex1autono} is an autonomous system, we have that $\mathcal{H}$ is constant in the first variable, and so we can write $\mathcal{H}(t,z_1,z_3) = H(z_1,z_3)$ for all $t \in \mathbb{R}$. Because $\mathcal{W}_{\loc}^c$ is locally invariant, we must have that 
			\begin{align*}
				z_1(z_1^2 + H(z_1,z_3)^2 + 2H(z_1,z_3))\frac{\partial}{\partial z_1}H(z_1,z_3) 
				= (H(z_1,z_3)+1)(z_1^2 +H(z_1,z_3)^2 + 2H(z_1,z_3)).
			\end{align*}
			If $z_1^2 + H(z_1,z_3)^2 + 2H(z_1,z_3) \neq 0$, then $H$ must satisfy
			\begin{equation*}
				\frac{\partial}{\partial z_1}H(z_1,z_3) = \frac{1}{z_1}(H(z_1,z_3)+1), \quad H(0,0) = 0,
			\end{equation*}
			which has obviously no solution. Hence $z_1^2 + (H(z_1,z_3)+1)^2 = 1$ and so $H(z_1,z_3) = \sqrt{1-z_1^2} -1$ since $H(0,0) = 0$. Clearly $H$ is analytic on $(-1,1) \times \mathbb{R}$ since
			\begin{equation*}
				H(z_1,z_3) = \sum_{k=1}^\infty (-1)^k \binom{\frac{1}{2}}{k} z_1^{2k},
			\end{equation*}
			for all $(z_1,z_3) \in (-1,1) \times \mathbb{R}$ due to the Binomial series. Hence, $\mathcal{H}$ is analytic on $\mathbb{R} \times (-1,1) \times \mathbb{R}$, which proves the claim. Transforming the map $\mathcal{H}$ back into original $y$-coordinates using \eqref{eq:Ex1variables} shows that the center manifold $\mathcal{W}_{\loc}^c$ is parametrized as
			\begin{align} \label{eq:Ex1inveq}
				(y_1(t) + \cos(t))^2 &+ (y_2(t)+\sin(t))^2 = 1, \quad
				y_3(t) = c_3, \quad c_3 \in \mathbb{R}.
			\end{align}
			Writing this back into $x$-coordinates yields
			\begin{align*}
				(x_1(t),x_2(t),x_3(t))
				= (y_1(t) + \cos(t), y_2(t)+ \sin(t), c_3), \quad c_3 \in \mathbb{R},
			\end{align*}
			but due to \eqref{eq:Ex1inveq} we have that $x_1(t)^2 + x_2(t)^2 = 1$ and so
			\begin{equation*}
				\mathcal{W}_{\loc}^c(\Gamma) = \{(x_1,x_2,x_3) \in \mathbb{R}^3 : x_1^2 + x_2^2 = 1 \}.
			\end{equation*}
			Notice that $\mathcal{W}_{\loc}^c$ and $\mathcal{W}_{\loc}^c(\Gamma)$ do not depend on a choice of $\delta > 0$. The reason is clear as for example the cylinder $\mathcal{W}_{\loc}^c(\Gamma)$ is an invariant manifold of \eqref{eq:Ex1} since the function $V : \mathbb{R}^3 \to \mathbb{R}$ defined by
			\begin{equation*}
				V(x_1,x_2,x_3) := x_1^2 + x_2^2 - 1 
			\end{equation*}
			is constant along the trajectories whose points are contained in $\mathcal{W}_{\loc}^c(\Gamma) = V^{-1} (\{0\})$.
		\end{proof}
		
		\begin{example} \label{ex:mobius}
			The analytic system
			\begin{equation} \label{eq:Exmobius}
				\begin{dcases}
					\dot{x}_1 = - x_2 + x_1 \Phi(x_1,x_2), \\
					\dot{x}_2 = x_1 + x_2 \Phi(x_1,x_2),\\
					\dot{x}_3 = \frac{1}{4}(1-\sigma x_2)(x_1^2 + x_2^2 - 1) + \frac{\sigma x_3}{2}(1+x_1),
				\end{dcases}
			\end{equation}
			where 
			\[
			\Phi(x_1,x_2) := \frac{\sigma}{4}(1-x_1)(x_1^2 + x_2^2 - 1) - \frac{x_3}{2}(1+\sigma x_2),
			\]
			admits for $\sigma \neq 0$ a $2 \pi$-periodic two-dimensional $C^k$-smooth center manifold that is locally diffeomorphic to a Möbius band for every $k \geq 1$. If $\sigma = 0$, then the $2\pi$-periodic center manifold is the whole state space $\mathbb{R}^3$ and \eqref{eq:Exmobius} admits a family of invariant tori.
		\end{example}

		\begin{proof}
			Notice that \eqref{eq:Ex1} admits for all $\sigma \in \mathbb{R}$ a $2\pi$-periodic solution $\gamma_\sigma(t) = (\cos(t),\sin(t),0)$. Hence, the solution of the variational equation around $\Gamma_\sigma$ is generated by the fundamental matrix $U_\sigma(t,s) = V_\sigma(t)V_\sigma(s)^{-1}$ where
			\begin{equation*}
				V_\sigma(t) = 
				\begin{pmatrix}
					\cos(t) \cos(\frac{t}{2}) & - \sin(t) & -e^{ \sigma t}\sin(\frac{t}{2})\cos(t) \\[1ex]
					\sin(t)\cos(\frac{t}{2}) & \cos(t) & -e^{ \sigma t}\sin(\frac{t}{2})\sin(t) \\[1ex]
					\sin(\frac{t}{2}) & 0 & e^{\sigma t} \cos(\frac{t}{2})     
				\end{pmatrix}.
			\end{equation*}
			The Floquet multipliers are given by $\lambda_{1} = 1, \lambda_{2} = -1$ and $\lambda_{3,\sigma} = -e^{2 \pi \sigma }$. Let $E_0^\sigma(t)$ and $E_{\pm}^\sigma(t)$ denote the center and (un)stable subspace (at time $t$) at parameter value $\sigma$ respectively. For the center subspace, we have that $E_{0}^\sigma(t) = \spn \{ \zeta_1(t),\zeta_2(t) \}$ for $\sigma \neq 0$ and $E_{0}^0(t) = \spn \{ \zeta_1(t),\zeta_2(t),\zeta_3(t) \}$. For the (un)stable subspace we obtain $E_{-}^{\sigma}(t) = \spn \{ \zeta_3(t) \}$ for $\sigma < 0$ and $E_{+}^{\sigma}(t) = \spn \{ \zeta_3(t) \}$ for $\sigma > 0$ where
			\begin{align*}
				\zeta_1(t) &= (-\sin(t),\cos(t),0), \\
				\zeta_2(t) &= \bigg( \cos(t)\cos(\frac{t}{2}), \sin(t)\cos(\frac{t}{2}), \sin(\frac{t}{2}) \bigg), \\
				\zeta_3(t) &= \bigg(-\cos(t)\sin(\frac{t}{2}),-\sin(t)\sin(\frac{t}{2}), \cos(\frac{t}{2}) \bigg).
			\end{align*}
			Notice that the eigenvector $\zeta_3(t)$ is perpendicular to the plane spanned by $\zeta_1(t)$ and $\zeta_2(t)$. Let us first discuss the case $\sigma \neq 0$. Observe that the center bundle $E_{0}^\sigma$ at parameter value $\sigma$ parametrizes locally a M\"obius band as a ruled surface since
			\begin{equation*}
				(x_1(t,v),x_2(t,v),x_3(t,v)) = \gamma_\sigma(t) + v \zeta_2(t),
			\end{equation*}
			for all $t \in \mathbb{R}$ and $v \in [-1,1]$. \Cref{thm:main} provides us for any $k\geq 1$ a $2\pi$-periodic $C^k$-smooth two-dimensional locally invariant center manifold $\mathcal{W}_{\loc}^c(\Gamma_\sigma)$ for \eqref{eq:Exmobius} around $\Gamma_\sigma$ tangent to the center bundle $E_{0}^\sigma$ that is locally diffeomorphic to a M\"obius band, see \Cref{fig:cmmobius}. 
			
			When $\sigma = 0$, it is clear that the Floquet multipliers are all on the unit circle where $1$ is simple and $-1$ has algebraic multiplicity $2$. Hence, the $2\pi$-periodic center manifold is $3$-dimensional, i.e. the whole state space $\mathbb{R}^3$. Moreover, \eqref{eq:Exmobius} admits a family of invariant tori $\{\mathbb{T}_l : l \geq 0 \}$ at $\sigma = 0$ with major radius $1$ and minor radius $r_l$ since the function $V_{l} : \mathbb{R}^3 \to \mathbb{R}$ defined by
			\begin{align*}
				V_{l}(x_1,x_2,x_3) := \bigg(\sqrt{x_1^2 + x_2^2} - 1\bigg)^2 + x_3^2 - r_l^2\bigg(\sqrt{x_1^2 + x_2^2}\bigg)
			\end{align*}
			where
			\begin{equation*}
				r_l^2(u) := l + \frac{1}{2}u(u-4)+\ln(u) + \frac{3}{2},
			\end{equation*}
			is constant along the trajectories whose points are contained in $\mathbb{T}_l := V_{l}^{-1}(\{0\})$. We claim that the family of tori is rooted at the cycle, i.e. $\mathbb{T}_0 = \Gamma_0$. It is clear that solving $V_0(x_1,x_2,x_3) = 0$ is equivalent to
			\begin{equation*}
				\frac{1}{2}u^2 - \ln(u) = \frac{1}{2} - x_3^2, \quad u = \sqrt{x_1^2 + x_2^2},
			\end{equation*}
			which has only one real solution at $x_3 = 0$ since $\frac{1}{2}u^2 - \ln(u) \geq \frac{1}{2}$ for all $u \geq 0$. Clearly this solution corresponds to the cycle $\Gamma_0$. Examples of such invariant tori can be found in \Cref{fig:cmmobius}.
		\end{proof}
		
		\begin{figure}[ht]
			\centering
			\includegraphics[width = 14cm]{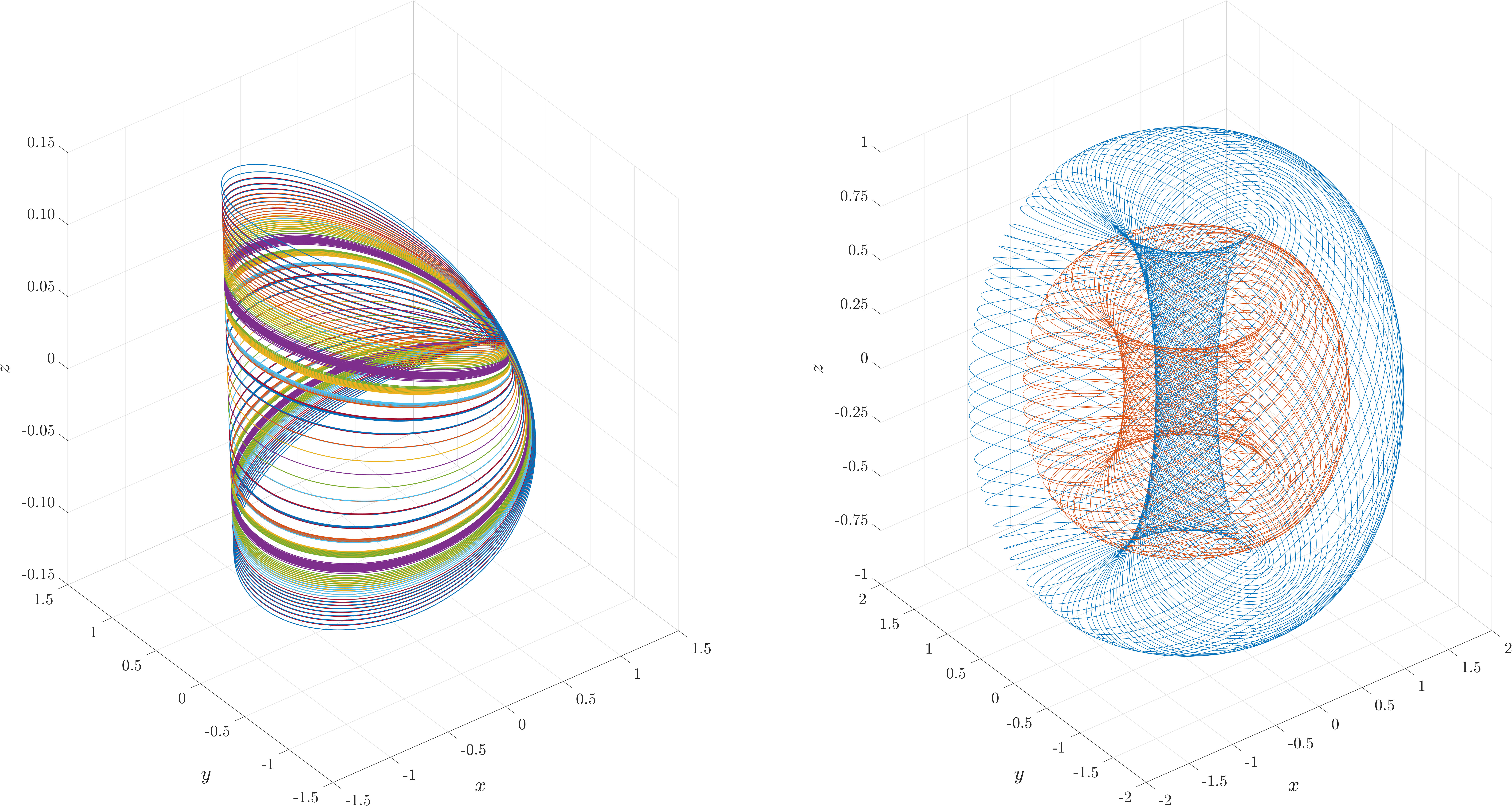}
			\caption{The left figure represents several forward orbits on a local $2\pi$-periodic two-dimensional center manifold (M\"obius band) around $\Gamma_{\sigma}$ for \eqref{eq:Exmobius} at parameter value $\sigma = -1$. The right figure represents two forward orbits on two different invariant tori for \eqref{eq:Exmobius} at parameter value $\sigma = 0$. The forward orbits are obtained by numerical integration and each orbit is represented by different color.}
			\label{fig:cmmobius}
		\end{figure}

		\begin{example} \label{ex:nonanalytic}
			The analytic system \eqref{eq:Ex2} admits a non-analytic $2\pi$-periodic center manifold.
		\end{example}
		\begin{proof}
			The fundamental matrix for the linearization around the origin of \eqref{eq:Ex2} reads 
			\begin{equation*}
				U(t,s) =
				\begin{pmatrix}
					1 & 0 \\
					0 & e^{s-t}\\
				\end{pmatrix}
			\end{equation*}
			for all $(t,s) \in \mathbb{R}^2$. The Floquet multipliers are given by $\lambda_1 = 1$ and $\lambda_2 = e^{-2\pi}$ and the center space (at time $t$) is given by $E_0(t) = \spn \{(1,0)\}$ while the stable space (at time $t$) is given by $E_{-}(t) = \spn \{(0,1)\}$. Hence, the $x$-axis corresponds to the center space and so the center manifold can be parametrized by $y(t) = \mathcal{H}(t,x)$, where $\mathcal{H}$ is $2\pi$-periodic in the first argument and consists solely of nonlinear terms. Because the center manifold is locally invariant, the map $\mathcal{H}$ must satisfy
			\begin{equation} \label{eq:Ex2invariance}
				\frac{\partial \mathcal{H}}{\partial t}(t, x) - x^2 \frac{\partial \mathcal{H}}{\partial x}(t, x) = -\mathcal{H}(t, x) + \sin(t)x^2.
			\end{equation}
			Assume that $\mathcal{H}$ is analytic on an open neighborhood of $\mathbb{R} \times \{ 0 \}$, then we can write locally $\mathcal{H}(t,x) = \sum_{n \geq 2} a_n(t)x^n$ for $2\pi$-periodic functions $a_n$. Filling this expansion into \eqref{eq:Ex2invariance} and comparing terms in $x^n$ shows that the $2\pi$-periodic functions $a_n$ must satisfy
			\begin{equation*}
				\begin{dcases}
					\dot{a}_2(t) + a_2(t) = \sin(t), \quad &n = 2, \\
					\dot{a}_n(t) + a_n(t) = (n-1)a_{n-1}(t), \quad &n \geq 3.
				\end{dcases}
			\end{equation*}
			Hence, $a_2(t) =  \alpha_2 \sin(t) + \beta_2 \cos(t)$, where $\alpha_2 = \frac{1}{2}$ and $\beta_2 = -\frac{1}{2},$ and
			\begin{align*}
				a_n(t) = \frac{(n-1)e^{-t}}{e^{2\pi}-1} \bigg( \int_t^{2\pi} e^{\tau} a_{n-1}(\tau) d\tau
				+ e^{2\pi} \int_0^t e^{\tau} a_{n-1}(\tau) d\tau \bigg).
			\end{align*}
			Let us prove by induction for $n \geq 2$ that $a_n$ is a linear combination of sines and cosines. If $n = 2$, then the result is clear. Assume that the claim holds for a certain $n \geq 3$, it follows from the induction hypothesis and applying integration by parts twice on both integrals that
			\begin{align*}
				a_{n}(t)&= \frac{(n-1) e^{-t}}{e^{2\pi}-1} \bigg( \int_t^{2\pi} e^{\tau} (\alpha_{n-1} \sin(\tau) 
				+ \beta_{n-1} \cos(\tau)) d\tau + e^{2\pi} \int_0^t e^{\tau} (\alpha_{n-1} \sin(\tau)
				+ \beta_{n-1} \cos(\tau)) d\tau \bigg) \\
				&=  \frac{n-1}{2}((\alpha_{n-1} + \beta_{n-1}) \sin(t) 
				+ (-\alpha_{n-1} + \beta_{n-1})\cos(t)),
			\end{align*}  
			which proves the claim. From the proof of the induction step, we obtain
			\begin{equation*}
				\begin{pmatrix}
					\alpha_n \\
					\beta_n
				\end{pmatrix}
				= \frac{n-1}{2}
				\begin{pmatrix}
					1 & 1 \\
					-1 & 1
				\end{pmatrix}
				\begin{pmatrix}
					\alpha_{n-1} \\
					\beta_{n-1}
				\end{pmatrix},
				\quad n \geq 3.
			\end{equation*}
			This is a linear system of difference equation and can be solved explicitly by computing the diagonalization of the associated matrix. The final result reads
			\begin{equation*}
				\begin{pmatrix}
					\alpha_n \\
					\beta_n
				\end{pmatrix}
				= \frac{(n-1)!}{2^{\frac{n-1}{2}}}
				\begin{pmatrix}
					\cos(\frac{(n-1)\pi}{4}) \\
					- \sin(\frac{(n-1)\pi}{4})
				\end{pmatrix}.
			\end{equation*}
			Hence, the $2\pi$-periodic functions $a_n$ are given by
			\begin{align*}
				a_n(t) = \frac{(n-1)!}{2^{\frac{{n-1}}{2}}}\bigg(\cos(\frac{(n-1)\pi}{4})\sin(t)
				- \sin(\frac{(n-1)\pi}{4}) \cos(t) \bigg), \quad n \geq 2.
			\end{align*}
			Using the angle addition and subtractions formula for cosines yields the center manifold expansion
			\begin{equation} \label{eq:Ex2H}
				\mathcal{H}(t,x) = \sum_{n \geq 2} \frac{(n-1)!}{2^{\frac{{n-1}}{2}}} \sin(t - \frac{(n-1)\pi}{4}) x^n.
			\end{equation}
			We will prove that the radius of convergence $R(t)$ at time $t \in \mathbb{R}$ of \eqref{eq:Ex2H} is zero. Let us first observe that for any  $n \geq 2$ one has
			\begin{equation*}
				\sup_{k \geq n} \bigg(k!\bigg|\sin(t-\frac{k \pi}{4}) \bigg| \bigg)^{\frac{1}{k}} \geq ((4n)! | \sin(t) |)^{\frac{1}{4n}},
			\end{equation*}
			for $t \neq l \pi$ and $\ l \in \mathbb{Z}$. To bound this supremum from below when $t = l \pi$ for some $l \in \mathbb{Z}$, choose $m \in \mathbb{Z}$ such that $r = 2(l-1-m) \geq n$ because then
			\begin{equation*}
				\sup_{k \geq n} \bigg(k!\bigg|\sin(l\pi-\frac{k \pi}{4}) \bigg| \bigg)^{\frac{1}{k}} \geq (r!)^{\frac{1}{r}} \geq (n!)^{\frac{1}{n}}.
			\end{equation*}
			The Cauchy-Hadamard theorem tells us that
			\begin{align*}
				\frac{1}{R(t)} = \limsup\limits_{n \to \infty}|a_n(t)|^{\frac{1}{n}} 
				\geq \sqrt{2} \min\{ \lim_{n \to \infty} (n! | \sin(t) |)^{\frac{1}{n}}, \lim_{n \to \infty} (n!)^{\frac{1}{n}} \} 
				= \infty,
			\end{align*}
			where in the first argument of the minimum it is assumed that $t \neq l \pi$ for all $l \in \mathbb{Z}$. This proves $R(t) = 0$ for all $t \in \mathbb{R}$, i.e. $\mathcal{H}$ is not analytic.
		\end{proof}
		
		\begin{example} \label{ex:nonunique}
			The analytic system \eqref{eq:Ex2} admits a locally (non)-unique family of $2\pi$-periodic center manifolds.    
		\end{example}
		\begin{proof}
			Recall from \Cref{ex:nonanalytic} that the parametrization $\mathcal{H}$ of the center manifold must satisfy \eqref{eq:Ex2invariance} in an open neighborhood of $\mathbb{R} \times \{0\}$. To construct the map $\mathcal{H}$ explicitly, let us first introduce (formally) for arbitrary constants $\alpha,\beta \in \mathbb{R}$ the family of functions $I_{\alpha,\beta} : \mathbb{R} \times \mathbb{R} \setminus \{0\} \to \mathbb{R}$ as
			\begin{align} \label{eq:Ialphabeta}
				I_{\alpha,\beta}(t,x) := - \sqrt{2} e^{-\frac{1}{x}}\bigg( \cos(\frac{1}{x}-t) I_{\alpha}^1(x)  
				+ \sin(\frac{1}{x}-t) I_\beta^2(x)  \bigg),
			\end{align}
			where the functions $I_\alpha^1$ and $I_\beta^2$ are defined by
			\begin{align} \label{eq:Ialphabeta1}
				I_\alpha^1(x) := \int_{\alpha}^x \frac{e^{\frac{1}{s}}}{s}\sin(\frac{1}{s} + \frac{\pi}{4})ds, \quad
				I_\beta^2(x) := \int_{\beta}^x \frac{e^{\frac{1}{s}}}{s}\sin(\frac{1}{s} - \frac{\pi}{4}) ds.
			\end{align}
			It turns out that the higher order derivatives of $I_{\alpha,\beta}(t,\cdot)$ will be important for the construction of the map $\mathcal{H}$. Therefore, let us first determine all values of $\alpha$ and $\beta$ for which \eqref{eq:Ialphabeta} is well-defined on $\mathbb{R} \times (-\infty,0)$. Clearly $I_{\alpha,\beta}$ is ill-defined on $\mathbb{R} \times (-\infty,0)$ whenever $\alpha,\beta > 0$ due to the singularities at zero for the functions defined in \eqref{eq:Ialphabeta1}. Hence, we must have that $\alpha,\beta \leq 0$. We will show that $\alpha = \beta = 0$ are the only values for which $I_{\alpha,\beta}$ is well-defined on $\mathbb{R} \times (-\infty,0)$. Notice that
			\begin{equation} \label{eq:limI0}
				\bigg|e^{-\frac{1}{x}} \cos(\frac{1}{x}-t)I_0^1(x) \bigg| \leq e^{-\frac{1}{x}} \int_x^0 \frac{e^{\frac{1}{s}}}{s} ds \to 0,
			\end{equation}
			as $x \uparrow 0$ by an application of L'H\^opital's rule. A similar computation for the second term in \eqref{eq:Ialphabeta} shows that $I_{0,0}$ is well-defined. To show that the function $I_{\alpha,\beta}$ is ill-defined on $\mathbb{R} \times (-\infty,0)$ for all $\alpha,\beta < 0$, consider for a fixed $t \in \mathbb{R}$ the sequence $(x_m)_{m \geq m_0}$ defined by $x_m := \frac{1}{t - m \pi}$, where the integer $m_0 \geq 0$ is chosen large enough to guarantee that $t - m_0 \pi < 0$. Hence,
			\begin{align} \label{eq:xmintegral}
				e^{-\frac{1}{x_m}} \cos(\frac{1}{x_m}-t)I_\alpha^1(x_m) =  
				(-1)^m e^{-\frac{1}{x_m}} \int_\alpha^{x_m} \frac{e^{\frac{1}{s}}}{s} \sin(\frac{1}{s} + \frac{\pi}{4}) ds.
			\end{align}
			Because the integrand is continuous and bounded above on $[\alpha,x_m] \subset [\alpha,0)$ for large enough $m \geq m_0$, it can be extended from the left continuously at zero such that it attains the value $M_{\alpha} < \infty$. If we set the integral in \eqref{eq:xmintegral} to be $M_\alpha^m$, then $M_\alpha^m \to M_\alpha$ when $m \to \infty$. Hence,
			\begin{equation*}
				e^{-\frac{1}{x_m}} \cos(\frac{1}{x_m}-t)I_\alpha^1(x_m) = e^{-\frac{1}{x_m}} (-1)^m M_\alpha^m
			\end{equation*}
			which is undetermined when $m \to \infty$. A similar reasoning shows that the second term in \eqref{eq:Ialphabeta} is ill-defined on $\mathbb{R} \times (-\infty,0)$ when $\beta < 0$ and so $I_{\alpha,\beta}$ is only well-defined on $\mathbb{R} \times (-\infty,0)$ whenever $\alpha = \beta = 0$. It can be proven similarly as in \eqref{eq:limI0} that $I_{\alpha,\beta}$ is well-defined on $\mathbb{R} \times (0,\infty)$ and that $\lim_{x \downarrow 0} I_{\alpha,\beta}(\cdot,x) = 0$ for all $\alpha,\beta > 0$.
			
			Our next goal is to determine the higher order partial derivatives of the second component of $I_{\alpha,\beta}$ evaluated at zero. A straightforward computation shows already that
			\begin{equation} \label{eq:partialIab}
				\frac{\partial}{\partial x} I_{\alpha,\beta}(t,x) = \frac{\sqrt{2}}{x^2} \bigg(I_{\alpha,\beta} \bigg(t+\frac{\pi}{4},x \bigg) - x \sin(t+\frac{\pi}{4}) \bigg).
			\end{equation}
			Let us write $I_{\alpha,\beta}(t,x) = \sum_{n=0}^N b_n(t)x^n + R_N(t,x)$ as a Taylor polynomial where $R_N$ is the remainder for some $N \in \mathbb{N}$. Filling in this Taylor polynomial into \eqref{eq:partialIab}, we see that $b_0$ is the zero function, $b_1(t) = \sin(t)$ and $b_{n}(t) = \frac{n-1}{\sqrt{2}}b_{n-1}(t-\frac{\pi}{4})$ for all $n=2,\dots,N$. This recurrence relation shows that
			\begin{align} \label{eq:partialsIab}
				\frac{\partial^n}{\partial x^n} I_{\alpha,\beta}(t,0) = n! b_n(t) = n!\frac{(n-1)!}{2^{\frac{n-1}{2}}} \sin(t-\frac{(n-1)\pi}{4}),
			\end{align}
			for all $n=1,\dots,N$, where $N$ can be taken arbitrary large. 
			
			The construction of the map $\mathcal{H}$ will consists of two parts, namely $x \in (-\infty,0]$ and $x \in [0,\infty)$. For the first part, let $\phi : \mathbb{R} \to \mathbb{R}$ be any $2\pi$-periodic differentiable function and observe that the map $\mathcal{H}_{-} : \mathbb{R} \times (-\infty,0] \to \mathbb{R}$ defined by 
			\begin{equation*}
				\mathcal{H}_{-}(t,x)
				:=
				\begin{dcases}
					e^{-\frac{1}{x}} \phi \bigg(\frac{1}{x}-t \bigg) + I_{0,0}(t,x) - \sin(t)x, \quad &t \in \mathbb{R}, \ x \in (-\infty, 0),\\
					0, \quad & t \in \mathbb{R}, \ x = 0,
				\end{dcases}
			\end{equation*}
			satisfies the local invariance equation \eqref{eq:Ex2invariance}. However, for $\mathcal{H}_{-}$ to be a parametrization of a center manifold on $\mathbb{R} \times (-\infty,0]$, we must have that $\lim_{x \uparrow 0} \mathcal{H}_{-}(\cdot,x) = 0$. Since $\lim_{x \uparrow 0} I_{0,0}(\cdot,x)$, it is clear that
			\begin{equation*}
				e^{-\frac{1}{x}} \phi \bigg(\frac{1}{x}-t \bigg) \to 0, \quad x \uparrow 0.
			\end{equation*}
			We claim $\phi$ must be the zero function. Consider for fixed $t \in \mathbb{R}$ and $r \in \mathbb{Q}$ the sequence $(y_{m})_{m \geq m_1}$ defined by $y_m := \frac{1}{t-r-2m\pi}$ where the integer $m_1 \geq 0$ is chosen large enough to guarantee that $t-r-2m_1\pi < 0$. The $2\pi$-periodicity of $\phi$ implies that
			\begin{equation*}
				e^{-\frac{1}{y_m}} \phi \bigg(\frac{1}{y_m}-t \bigg) = \phi(r) e^{-\frac{1}{y_m}} \to \infty, \quad m \to \infty,
			\end{equation*}
			unless $\phi(r) = 0$. As $r \in \mathbb{Q}$ is arbitrary we have that $\phi$ is the zero function on $\mathbb{Q}$ and because $\phi$ is (at least) continuous, we have that $\phi$ is the zero function on $\mathbb{R}$. Moreover, we obtain from \eqref{eq:partialsIab} directly that $\lim_{x \uparrow 0} \frac{\partial}{\partial x} \mathcal{H}_{-}(\cdot,x) = 0$ and so $\mathcal{H}_{-}$ is indeed a parametrization of a center manifold on $ \mathbb{R} \times (-\infty,0]$. In addition, it follows directly from \eqref{eq:partialIab} that $\mathcal{H}_{-}$ is $C^{\infty}$-smooth on $\mathbb{R} \times (-\infty,0]$.
			
			For the second part, let $\phi : \mathbb{R} \to \mathbb{R}$ be any $2\pi$-periodic $C^k$-smooth function and observe that for any $\alpha,\beta > 0$ and sufficiently small $\delta_k > 0$ the map $\mathcal{H}_{+,\phi}^{\alpha,\beta} : \mathbb{R} \times [0,\delta_k) \to \mathbb{R}$ defined by
			\begin{equation*}
				\mathcal{H}_{+,\phi}^{\alpha,\beta}(t,x)
				:=
				\begin{dcases}
					e^{-\frac{1}{x}} \phi \bigg(\frac{1}{x}-t \bigg) + I_{\alpha,\beta}(t,x) - \sin(t)x, \quad &t \in \mathbb{R}, \ x \in (0, \delta_k),\\
					0, \quad & t \in \mathbb{R}, \ x = 0,
				\end{dcases}
			\end{equation*}
			satisfies the local invariance equation \eqref{eq:Ex2invariance}. Since $\phi$ is $C^k$-smooth and $2\pi$-periodic, we have for any $l =0,\dots,k$ that its $l$th derivative is bounded above by some real number $0 < M_l < \infty$. Hence,
			\begin{equation*}
				\bigg| e^{-\frac{1}{x}} \phi \bigg(\frac{1}{x}-t \bigg) \bigg| \leq M_{0} e^{-\frac{1}{x}} \to 0, \quad x \downarrow 0,
			\end{equation*}
			which already proves that $\lim_{x \downarrow 0} \mathcal{H}_{+,\phi}^{\alpha,\beta}(\cdot,x) = 0$. To prove that $\mathcal{H}_{+,\phi}^{\alpha,\beta}$ is tangent to the center bundle and $C^k$-smooth (at the origin), note for any $l=0,\dots,k$ that
			\begin{equation*}
				\bigg| \frac{d^l}{dx^l} \bigg[e^{-\frac{1}{x}} \phi \bigg(\frac{1}{x}-t \bigg) \bigg] \bigg| \leq e^{-\frac{1}{x}}\sum_{q=0}^l p_q \bigg(\frac{1}{|x|} \bigg) \to 0,
			\end{equation*}
			as $x \downarrow 0$ due to the general Leibniz rule, Fa\`{a} di Bruno's formula and the fact that $p_q$, dependent on $M_0,\dots,M_q$, is a polynomial for all $q=0,\dots,l$. Hence, $\mathcal{H}_{+,\phi}^{\alpha,\beta}$ is a $C^k$-smooth function on $\mathbb{R} \times [0,\delta_k)$ for some $\delta_k > 0$. As a consequence of the results derived above, the map $\mathcal{H}_{\phi}^{\alpha,\beta} : \mathbb{R} \times (-\infty,\delta_k)$ defined by
			\begin{align} \label{eq:Habphi}
				\mathcal{H}_{\phi}^{\alpha,\beta}(t,x) := 
				\begin{dcases}
					\mathcal{H}_{-}(t,x), &t \in \mathbb{R}, \ x \in (-\infty, 0],\\
					\mathcal{H}_{+,\phi}^{\alpha,\beta}(t,x), &t \in \mathbb{R}, \ x \in [0,\delta_k),\\
				\end{dcases}
			\end{align}
			parametrizes a locally (non)-unique family of $2\pi$-periodic $C^k$-smooth center manifolds around $\mathbb{R} \times \{0\}$ of \eqref{eq:Ex2}. Two different $2\pi$-periodic center manifolds for \eqref{eq:Ex2} are visualized in \Cref{fig:cmexample}.
		\end{proof}
		
		\begin{figure}[ht]
			\centering
			\includegraphics[width = 14cm]{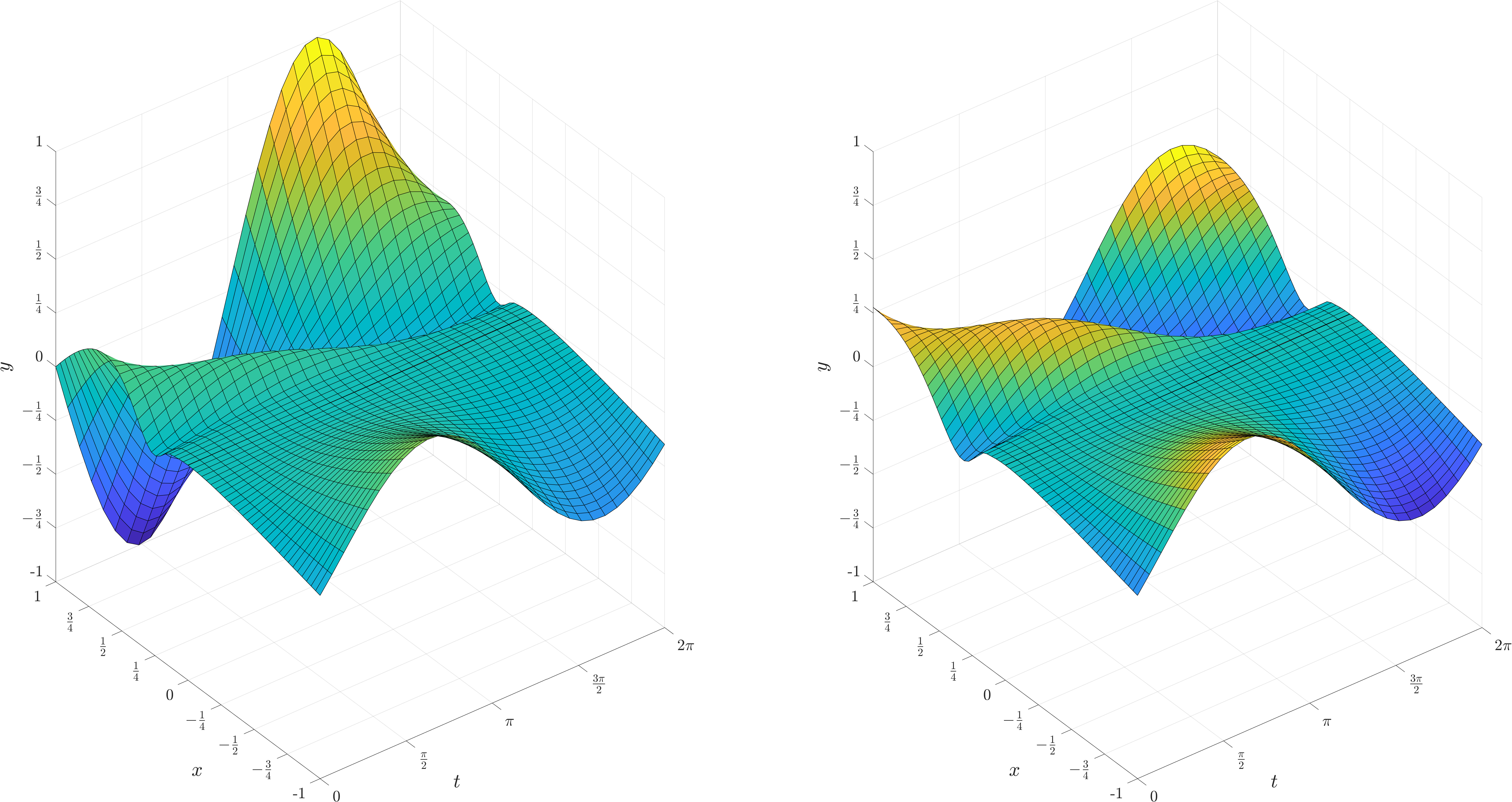}
			\caption{Two different $C^{\infty}$-smooth non-analytic $2\pi$-periodic center manifolds around $\mathbb{R} \times \{0\}$ for the analytic system \eqref{eq:Ex2} parametrized by the maps $\mathcal{H}_{0}^{1,1}$ (left) and $\mathcal{H}_{0}^{2,3}$ (right), respectively.}
			\label{fig:cmexample}
		\end{figure}
		
		\begin{example} \label{ex:Cinftysmooth}
			The analytic system \eqref{eq:Ex2} admits a $2\pi$-periodic $C^{\infty}$-smooth center manifold.
		\end{example}
		\begin{proof}
			The map $\mathcal{H}_{0}^{1,1}$ from \eqref{eq:Habphi} provides us a $2\pi$-periodic  $C^{\infty}$-smooth center manifold for \eqref{eq:Ex2} on $(-\infty,\delta_\infty)$ with $\delta_\infty > 0$ since $\mathcal{H}_{0}^{1,1}$ is $C^{\infty}$-smooth in an open neighborhood of $\mathbb{R} \times \{0\}$.
		\end{proof}
		
		\begin{example} \label{ex:nonCinfitysmooth}
			The analytic system \eqref{eq:Exshrinking}
			admits for any $k \geq 0$ a $2\pi$-periodic $C^k$-smooth center manifold, but not a $2\pi$-periodic $C^{\infty}$-smooth center manifold.
			\begin{proof}
				The fundamental matrix for the linearization around the origin of \eqref{eq:Exshrinking} reads 
				\begin{equation*}
					U(t,s) =
					\begin{pmatrix}
						1 & 0 & 0 \\
						0 & e^{t-s} & 0\\
						0 & 0 & 1
					\end{pmatrix}
				\end{equation*}
				for all $(t,s) \in \mathbb{R}^2$. The Floquet multipliers are given by $\lambda_1 = 1, \lambda_2 = e^{2\pi}$ and $\lambda_3 = 1$. The center space (at time $t$) is given by $E_0(t) = \spn \{(1,0,0),(0,0,1)\}$ while the unstable space (at time $t$) is given by $E_{+}(t) = \spn \{(0,1,0)\}$. Hence, the $xz$-plane corresponds to the center space and so the center manifold can be parametrized by $y(t) = \mathcal{H}(t,x,z)$, where $\mathcal{H}$ is $2\pi$-periodic in the first argument and only consists of nonlinear terms. It follows from \Cref{thm:main} that there exists for any $k \geq 1$ an open neighborhood $U_{2k}$ of $\mathbb{R} \times \{0\} \times \{0\}$ such that the map $\mathcal{H}$ is $C^{2k}$-smooth on $U_{2k}$. Hence, one can write
				\begin{equation} \label{eq:Hexpansion}
					\mathcal{H}(t,x,z) = \sum_{n=2}^{2k} a_n(t,z)x^{n} + \mathcal{O}(x^{2k+1})
				\end{equation}
				in the neighborhood $U_{2k}$. Because the center manifold is locally invariant, the map $\mathcal{H}$ must satisfy
				\begin{align} \label{eq:Ex3inveq}
					\frac{\partial \mathcal{H}}{\partial t}(t,x,z) + x(z-x^2) \frac{\partial \mathcal{H}}{\partial x}(t,x,z) = \mathcal{H}(t,x,z) + (1+\sin(t))x^2
				\end{align}
				Substituting \eqref{eq:Hexpansion} into \eqref{eq:Ex3inveq} and comparing terms in $x^n$ for $n = 2,\dots,2k$ shows that the functions $a_n$, which are $2\pi$-periodic in the first component, must satisfy
				\begin{equation} \label{eq:a2an}
					\begin{dcases}
						\frac{\partial a_2}{\partial t}(t,z) + (2z - 1) a_2(t,z) = 1 + \sin(t), \quad &n = 2, \\
						\frac{\partial a_n}{\partial t}(t,z) + (nz - 1) a_n(t,z) = (n-2)a_{n-2}(t,z), \quad &n = 3,\dots,2k.
					\end{dcases}
				\end{equation}    
				
				Because $\mathcal{H}$ consists only of nonlinear terms, we can assume that $a_1 = 0$. Solving the $2\pi$-periodic boundary value problem \eqref{eq:a2an} for $n=2$ yields
				\begin{align} \label{eq:a2}
					a_2(t,z) = \frac{1}{2z - 1} + \frac{2z-1}{(2z-1)^2 + 1} \sin(t) - \frac{1}{(2z-1)^2 + 1}\cos(t), \quad z \neq \frac{1}{2}.
				\end{align}
				Furthermore, at $z = \frac{1}{2}$ one verifies easily from \eqref{eq:a2an} that $a_2(\cdot,\frac{1}{2})$ does not admit a $2\pi$-periodic solution. Moreover, all odd coefficients $a_1,a_3,\dots,a_{2k-1}$ are zero and the even coefficients $a_2,a_4,\dots,a_{2k}$ are recursively given by
				\begin{align*}
					a_{2n}(t,z) &= \frac{2(n-1)e^{-(2nz-1)t}}{e^{2\pi(2nz-1)}-1} \bigg( \int_t^{2\pi} e^{(2nz-1)\tau} a_{2(n-1)}(\tau,z) d\tau \\
					&+ e^{2\pi(2nz-1)} \int_0^{t} e^{(2nz-1)\tau} a_{2(n-1)}(\tau,z) d\tau \bigg), \quad  z \neq \frac{1}{2n}.
				\end{align*}
				To obtain a semi-explicit representation for $a_{2n}$, we will prove by induction on $n = 1,2,\dots,k$ that 
				\begin{align} \label{eq:a2nindunction}
					a_{2n}(t,z) = 2^{n-1} (n-1)! \prod_{l=1}^{n} \frac{1}{2lz - 1}
					+ \alpha_n(z) \sin(t) + \beta_n(z) \cos(t), \quad z \neq \frac{1}{2n},  
				\end{align}
				where $\alpha_n$ and $\beta_n$ are well-defined rational functions on $\mathbb{R}$. Clearly, the claim holds for $n=1$ due to \eqref{eq:a2}. Assume that the claim holds for a certain $n \geq 2$. Along the same lines of the induction step in \Cref{ex:nonanalytic}, one derives
				\begin{align*}
					a_{2n}(t,z) = 2^{n-1} (n-1)! \prod_{l=1}^{n} \frac{1}{2lz - 1} &+  \frac{2(n-1)[(2nz-1)\alpha_{n-1}(z) + \beta_{n-1}(z)]}{(2nz-1)^2+1} \sin(t) \\
					&+ \frac{2(n-1)[(2nz-1)\beta_{n-1}(z) - \alpha_{n-1}(z)]}{(2nz-1)^2+1} \cos(t).
				\end{align*}
				It remains to show that the coefficients in front of the sine and cosine are well-defined rational functions on $\mathbb{R}$. Clearly,
				\begin{align*}
					\begin{pmatrix}
						\alpha_n(z) \\
						\beta_n(z)
					\end{pmatrix}
					= \frac{2(n-1)}{(2nz-1)^2+1}
					\begin{pmatrix}
						2nz-1 & 1 \\
						-1 & 2nz-1
					\end{pmatrix}
					\begin{pmatrix}
						\alpha_{n-1}(z) \\
						\beta_{n-1}(z)
					\end{pmatrix},
					\quad n \geq 3,
				\end{align*}
				with initial condition
				\begin{equation*}
					\alpha_2(z) = \frac{2z-1}{(2z-1)^2 + 1}, \quad \beta_2(z) = \frac{1}{(2z-1)^2 + 1}.
				\end{equation*}
				Solving this linear system of difference equations semi-explicitly yields
				\begin{align*}
					&\begin{pmatrix}
						\alpha_n(z) \\
						\beta_n(z)
					\end{pmatrix}
					= 2^{n-1} (n-1)!
					\bigg[\prod_{l=2}^n \frac{1}{(2lz-1)^2 + 1} 
					\begin{pmatrix}
						2lz-1 & 1 \\
						-1 & 2lz-1
					\end{pmatrix}
					\bigg]
					\begin{pmatrix}
						\alpha_2(z) \\
						\beta_2(z)
					\end{pmatrix}.
				\end{align*}
				Hence, $\alpha_n$ and $\beta_n$ are both rational functions that are well-defined on $\mathbb{R}$ since $(2lz-1)^2 + 1 \geq 0$ for all $l=1,\dots,n$. This concludes the induction step. On the other hand, if $z = \frac{1}{2n}$, then one can verify rather easily from \eqref{eq:a2an} that $a_n(\cdot,\frac{1}{2n})$ has no $2\pi$-periodic solution. 
				
				Using \eqref{eq:a2nindunction} in combination with \eqref{eq:Hexpansion}, we see that $\mathcal{H}(t,x,\cdot)$ is not $C^{2k}$-smooth on $(-\frac{1}{2k},\frac{1}{2k}]$ since $a_{2k}(\cdot,\frac{1}{2k})$ is simply undefined. Suppose now that $\mathcal{H}$ is $C^{\infty}$-smooth on $\mathbb{R} \times \{0\} \times \{0\}$, then for fixed $t \in \mathbb{R}$ and non-zero $x \in \mathbb{R}$ there exists an $\varepsilon > 0$ such that $\mathcal{H}(t,x,\cdot)$ is $C^\infty$-smooth on $(-\varepsilon,\varepsilon)$. Now, if $k \geq 1$ is an integer that satisfies $k > \frac{1}{2 \varepsilon}$, then $\mathcal{H}(t,x,\cdot)$ is not $C^{2k}$-smooth on $(-\varepsilon,\varepsilon)$. This contradicts the assumption that \eqref{eq:Exshrinking} admits a $C^{\infty}$-smooth $2\pi$-periodic center manifold at the origin. To illustrate the cascade of singularities of the periodic center manifold towards the origin, second- and fourth-order approximations at different time steps of $\mathcal{H}$ are presented in \Cref{fig:CMsingularity}.
			\end{proof} 
		\end{example}
		
		\begin{figure}[ht]
			\centering
			\includegraphics[width = 14cm]{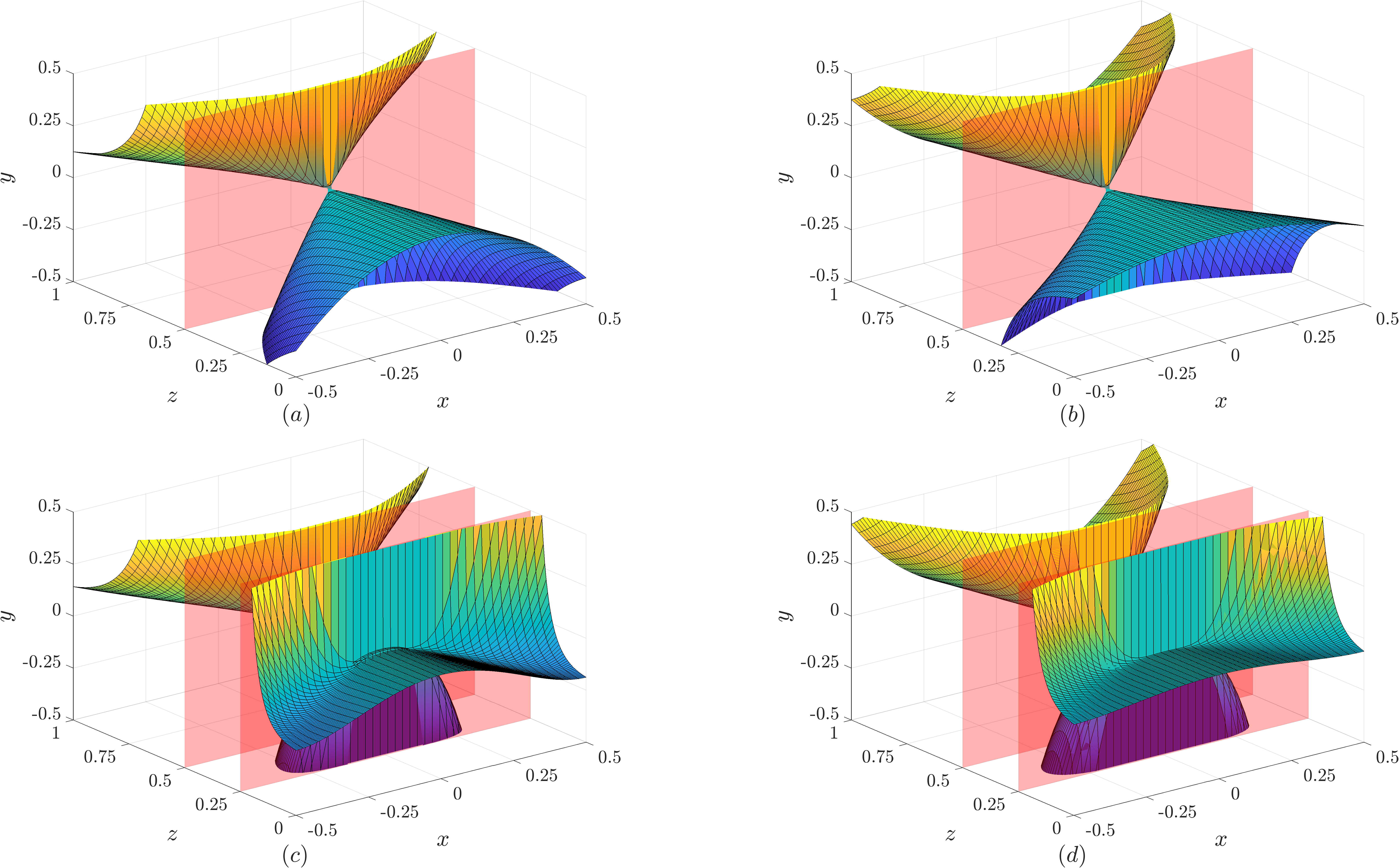}
			\caption{In $(a)$ a second-order approximation of $\mathcal{H}(0,\cdot,\cdot)$ and in $(b)$ a second-order approximation of $\mathcal{H}(\pi,\cdot,\cdot)$ . In $(c)$ a fourth-order approximation of $\mathcal{H}(0,\cdot,\cdot)$ and in $(d)$ a fourth-order approximation of $\mathcal{H}(\pi,\cdot,\cdot)$. The red vertical planes indicate the singularities at $z=\frac{1}{4}$ and $z = \frac{1}{2}$.}
			\label{fig:CMsingularity}
		\end{figure}
		
		\section{Conclusion and outlook}
		We have proven the existence of a periodic smooth locally invariant center manifold near a nonhyperbolic cycle in the setting of finite-dimensional ordinary differential equations. Our results are based on rather simple consequences of Floquet theory in combination with a fixed point argument on the easi\-ly available variation of constants formula for periodic (nonlinear) ODEs. In addition, we have provided several examples of (non)-unique, (non)-$C^\infty$-smooth and (non)-analytic periodic center manifolds to illustrate that periodic center manifolds admit similar interesting qualitative properties as center manifolds for equilibria.
		
		Despite our illustrations from \Cref{sec:examples} being very insightful on the nature of periodic center manifolds, it is not clear under which conditions a periodic center manifold is unique, non-unique or locally (non)-unique. To answer the first question, we believe that one must generalize techniques from \cite{Sijbrand1985} towards periodic center manifolds to state and prove a similar result as in \cite[Theorem 3.2]{Sijbrand1985}. Moreover, if a periodic center manifold is not uniquely determined, how much can two periodic center manifolds differ from each other? Such results have already been established in \cite[Section 4]{Sijbrand1985} for center manifolds for equilibria, but the question remains unanswered for periodic center manifolds. However, we have already seen in \Cref{ex:nonunique} that periodic center manifolds may differ from each other by a factor of $e^{-\frac{1}{x}}\phi(\frac{1}{x} - t)$, where $\phi$ is any $T$-periodic (at least) differentiable function. Furthermore, it is not clear under which conditions a $C^\infty$-smooth or analytic periodic center manifold may exist while this question is addressed and answered in \cite[Section 5 and 6]{Sijbrand1985} for center manifolds for equilibria. In particular, recall from \Cref{ex:nonanalytic} that the $C^\infty$-smooth periodic center manifold is not analytic for all $t \in \mathbb{R}$. However, is it possible to construct an example where a $C^{\infty}$-smooth periodic center manifold may change periodically from non-analytic to analytic? Or is there a possibility that the (non)-analyticitiy of a periodic center manifold is time-independent? 
		
		\section*{Acknowledgements}
		The authors would like to thank Prof. Renato Huzak (Hasselt University), Prof. Peter De Maesschalck (Hasselt University) and Dr. Heinz Han{\ss}mann (Utrecht University) for helpful discussions and suggestions.
		
		\appendix
		\section{Smoothness of the center manifold} \label{appendix:smoothness}
		In this appendix, we will prove that the map $\mathcal{C}$ inherits the same finite order of smoothness as the nonlinearity $R$. Our results are based on the theory of contraction on scales of Banach spaces, see \cite[Section IX.6 and Appendix IV]{Diekmann1995} and \cite{Vanderbauwhede1987,Hupkes2006,Hupkes2008,Church2018,Lentjes2023a} for applications of this theory to ordinary differential equations and (mixed) functional differential equations. Our arguments here are based on the strategy developed in the mentioned references and closely follow \cite{Lentjes2023a}.
		
		To prove additional smoothness of the map $\mathcal{C}$, let us first observe that we are only interested in pairs $(s,y_0) \in E_0$ due to \eqref{eq:mapC}. Therefore, let us incorporate the starting time $s$ inside the domain of the fixed point operator $\mathcal{G}_s^\eta$ from \Cref{subsec: Lipschitz CM}. Hence, define for $\eta \in (0,\min\{-a,b\})$ and sufficiently small $\delta > 0$ the map $\mathcal{G}^\eta$ by
		\begin{equation*}
			\BC_s^\eta \times E_0 \ni (u,s,y_0) \mapsto U_s^\eta y_0 + \mathcal{K}_s^\eta (\tilde{R}_\delta(u)) \in \BC_s^\eta,
		\end{equation*}
		and following the same steps from \Cref{subsec: Lipschitz CM}, we have that $\mathcal{G}^\eta(\cdot,s,y_0)$ has a unique fixed point $\hat{u}^\eta : E_0 \to \BC_s^\eta$ such that $\hat{u}^\eta(s,\cdot)$ is globally Lipschitz and satisfies $\hat{u}^\eta(s,0) = 0$ for all $s \in \mathbb{R}$. It turns out that the space $\BC_s^\eta$ is not really suited to increase smoothness of the center manifold. The main idea is to work with another $\eta$-exponent that makes a trade-off between ensuring smoothness while not losing the contraction property. To make this construction, choose an interval $[\eta_{-},\eta_{+}] \subset (0,\min\{-a,b\})$ such that $k \eta_{-} < \eta_{+}$ and $\delta > 0$ small enough to guarantee that
		\begin{equation} \label{eq:contractionregularity}
			L_{\delta}\|\mathcal{K}_s^\eta\|_{\eta,s} < \frac{1}{4}, \quad \forall \eta \in [\eta_{-},\eta_{+}], \ s \in \mathbb{R},
		\end{equation}
		which is possible since $L_{\delta} \to 0$ as $\delta \downarrow 0$ proven in \Cref{prop:global lipschitz}. 
		
		From this construction, it is clear that we would like to switch back and forth between, for example, the fixed points $\hat{u}^\eta(s,y_0)$ and $\hat{u}^{\eta_{-}}(s,y_0)$. Therefore, introduce for any $ 0 < \eta_1 \leq \eta_2 < \min\{-a,b\}$ the linear embedding $\mathcal{J}_s^{\eta_2,\eta_1} : \BC_s^{\eta_1} \hookrightarrow \BC_s^{\eta_2}$ and notice that this map is bounded since
		\begin{align*}
			\|u\|_{\eta_2, s} &= \sup_{t \in \mathbb{R}}e^{-\eta_2|t - s|}\|u(t)\| 
			\leq \sup_{t \in \mathbb{R}}e^{-\eta_1|t - s|}\|u(t)\|
			= \|u\|_{\eta_1, s} < \infty,
		\end{align*}
		for any $u \in \BC_s^{\eta_1}$. Hence, $\mathcal{J}_s^{\eta_2,\eta_1}$ is $C^{\infty}$-smooth and $\BC_s^{\eta_1}$ can be considered as a subspace of $\BC_s^{\eta_2}$. The following lemma shows we can switch back and forth between the fixed points of interest. 
		\begin{lemma}
			Let $0 < \eta_1 \leq \eta_2 < \min\{-a, b\}$ and $s \in \mathbb{R}$. Assume that $\hat{u}^{\eta_1}(s,y_0)$ is the fixed point of $\mathcal{G}^{\eta_1}(\cdot, s,y_0)$ for some $(s,y_0) \in E_0$. Then $\hat{u}^{\eta_2}(s,y_0) = \mathcal{J}_s^{\eta_2, \eta_1}\hat{u}^{\eta_1}(s,y_0)$.
		\end{lemma}
		
		\begin{proof}
			Note that the definition of the fixed point operator does not depend explicitly on the choice of $\eta \in (0,\min\{-a, b\})$, since $\mathcal{K}_{s}^{\eta_1} u = \mathcal{K}_{s}^{\eta_2} u$ for all $u \in \BC^{\eta_1}_{s}$. Then by uniqueness of the fixed point and since $\BC_s^{\eta_1}$ is continuously embedded in $\BC_s^{\eta_2}$, it is clear that $\hat{u}_{s}^{\eta_2}(y_0) = \mathcal{J}_s^{\eta_2, \eta_1}\hat{u}_{s}^{\eta_1}(y_0)$.
		\end{proof}
		
		A first step in increasing smoothness of the center manifold is to show that $\tilde{R}_{\delta}$ is sufficiently smooth. Recall from \Cref{subsec: modification nonlinearity} that $R_\delta$ is $C^k$-smooth. Consider now for any pair of integers $p,q \geq 0$ with $p + q \leq k$ the map $\Tilde{R}_\delta^{(p,q)}(u) \in \mathcal{L}^q(C(\mathbb{R},\mathbb{R}^n)) := \mathcal{L}^q(C(\mathbb{R},\mathbb{R}^n),C(\mathbb{R},\mathbb{R}^n))$ defined by
		\begin{align*}
			\Tilde{R}_{\delta}^{(p,q)}(u)(v_1,\dots,v_q)(t)
			:= D_1^p D_2^q R_\delta(t,u(t))(v_1(t),\dots,v_q(t)).
		\end{align*}
		Here $\mathcal{L}^q(Y,Z)$ denotes the space of $q$-linear mappings from $Y^q := Y \times \dots \times Y$ into $Z$ for Banach spaces $Y$ and $Z$. The following three lemmas, adapted from the literature towards the finite-dimensional ODE-setting, will be crucial in the proof of \Cref{thm:smoothnesscmt}.
		
		\begin{lemma}[{\cite[Lemma XII.7.3]{Diekmann1995} and \cite[Proposition 8.1]{Hupkes2008}}] 
			\label{lemma:smoothness3}
			Let $p,q \geq 0$ be integers with $p+q \leq k$ and $\eta \geq q \mu > 0$. Then for any $u \in C(\mathbb{R},\mathbb{R}^n)$ we have $\Tilde{R}_{\delta}^{(p,q)}(u) \in \mathcal{L}^q(\BC_s^\mu,\BC_s^\eta)$, where the norm is bounded by
			\begin{equation*}
				\| \Tilde{R}_{\delta}^{(p,q)} \| \leq \sup_{t \in \mathbb{R}} e^{-(\eta-q\mu)|t-s|} \| D_1^p D_2^q R_\delta(t,u(t)) \| < \infty.
			\end{equation*}
			Furthermore, consider any $0 \leq l \leq k - (p+q)$ and $\sigma > 0$. If $\eta > q \mu + l \sigma$, then the map $u \mapsto \Tilde{R}_{\delta}^{(p,q)}$ from $\BC_s^\sigma$ into $\mathcal{L}^q(\BC_s^\mu,\BC_s^\eta)$ is $C^l$-smooth with $D^l \Tilde{R}_{\delta}^{(p,q)} = \Tilde{R}_\delta^{(p,q+l)}. ~~~~~$
		\end{lemma}
		
		\begin{lemma}[{\cite[Lemma XII.7.6]{Diekmann1995} and \cite[Proposition 8.2]{Hupkes2008}}] 
			\label{lemma: smoothness4}
			Let $p,q \geq 0$ be integers with $p+q < k$ and let $\eta > q \mu + \sigma$ for some $\mu,\sigma > 0$. Consider a map $\Phi \in C^1(E_0,\BC_s^\sigma)$. Then the map $\Tilde{R}_{\delta}^{(p,q)} \circ \Phi : E_0 \to \mathcal{L}^q(\BC_s^\mu, \BC_s^\eta)$ is $C^1$-smooth with
			\begin{align*}
				D(\Tilde{R}_\delta^{(p,q)} \circ \Phi)(s_0,y_0)(v_1,\dots,v_q,(s_1,y_1)) =
				\Tilde{R}_\delta^{(p,q+1)}(\Phi(s_0,y_0))(v_1,\dots,v_q,D\Phi(s_0,y_0)(s_1,y_1)).~
			\end{align*}
			
		\end{lemma}
		
		\begin{lemma}[{\cite[Lemma XII.6.6 and XII.6.7]{Diekmann1995}}]
			\label{lemma:fixedpointsmooth}
			Let $Y_0,$ $Y,$ $Y_1,$ and $\Lambda$ be Banach spaces with continuous embeddings $J_0 : Y_0 \hookrightarrow Y$ and $J : Y \hookrightarrow Y_1$. Consider the fixed point problem $y = f(y,\lambda)$ for $f : Y \times \Lambda \to Y$. Suppose that the following conditions hold.
			\begin{enumerate}
				\item[$1.$] The function $g : Y_0 \times \Lambda \to Y_1$ defined by $ g(y_0,\lambda) := Jf(J_0y_0,\lambda)$ is of the class $C^1$ and there exist mappings $f^{2:} : J_0 Y_0 \times \Lambda \to \mathcal{L}(Y)$ and $f_1^{(1)} : J_0Y_0 \times \Lambda \to \mathcal{L}(Y_1)$ such that $D_1 g(y_0,\lambda) \xi = Jf^{(1)}(J_0y_0,\lambda)J_0$ for all $(y_0,\lambda,\xi) \in Y_0 \times \Lambda \times Y_0$ and $Jf^{(1)}(J_0y_0,\lambda)y = f_1^{(1)}(J_0y_0,\lambda)Jy$ for all $(y_0,\lambda,y) \in Y_0 \times \Lambda \times Y$.
				\item[$2.$] There exists a $\kappa \in [0,1)$ such that for all $\lambda \in \Lambda$ the map $f(\cdot,\lambda) : Y \to Y$ is Lipschitz continuous with Lipschitz constant $\kappa$, independent of $\lambda$. Furthermore, for any $\lambda \in \Lambda$ the maps $f^{(1)}(\cdot,\lambda)$ and $f_1^{(1)}(\cdot,\lambda)$ are uniformly bounded by $\kappa$.
				\item[$3.$] Under the previous condition, the unique fixed point $\Psi : \Lambda \to Y$ satisfies $\Psi(\lambda) = f(\Psi(\lambda),\lambda)$ and can be written as $\Psi = J_0 \circ \Phi$ for some continuous $\Phi : \Lambda \to Y_0$.
				\item[$4.$] The function $f_0 : Y_0 \times \Lambda \to Y$ defined by $f_0(y_0,\lambda) = f(J_0y_0,\lambda)$ has continuous partial derivative $D_2 f : Y_0 \times \Lambda \to \mathcal{L}(\Lambda,Y).$
				\item[$5.$] The mapping $Y_0 \times \Lambda \ni (y,\lambda) \mapsto J \circ f^{(1)}(J_0y,\lambda) \in \mathcal{L}(Y,Y_1)$ is continuous.
			\end{enumerate}
			Then the map $J \circ \Psi$ is of the class $C^1$ and $D(J \circ \Psi)(\lambda) = J \circ \mathcal{A}(\lambda)$ for all $\lambda \in \Lambda$, where $A = \mathcal{A}(\lambda) \in \mathcal{L}(\Lambda,Y)$ is the unique solution of the fixed point equation $A = f^{(1)}(\Psi(\lambda),\lambda) A + D_2 f_0(\Psi(\lambda),\lambda).~~~~~$
		\end{lemma}
		\par
		\medskip
		\noindent
		Now we can prove the main results of this appendix.
		\begin{proposition} \label{thm:smoothnesscmt}
			For each $l \in \{1,\dots,k\}$ and $\eta \in (l\eta_{-},\eta_{+}] \subset (0,\min\{-a,b\})$, the map $\mathcal{J}_s^{\eta,\eta_{-}} \circ \hat{u}^{\eta_{-}} : E_0 \to \BC_s^\eta$ is $C^l$-smooth provided that $\delta > 0$ is sufficiently small.
		\end{proposition}
		\begin{proof}
			To begin, we choose $\delta > 0$ small enough so that \eqref{eq:contractionregularity} holds. We prove the assertion by induction on $l$. Let $l=k=1$ and $\eta \in (\eta_{-},\eta_+]$ be given. We show that \Cref{lemma:fixedpointsmooth} applies with the Banach spaces $Y_0 = Y = \BC_s^{\eta_{-}}, Y_1 = \BC_s^\eta$ and $\Lambda = E_0$, and operators
			\begin{align*}
				f(u,s,y_0) &= \mathcal{G}^{\eta_{-}}(u,s,y_0), \\
				f^{(1)}(u,s,y_0) &= \mathcal{K}_s^{\eta_{-}} \circ \Tilde{R}_{\delta}^{(0,1)}(u), \\
				f_1^{(1)} &= \mathcal{K}_s^{\eta} \circ \Tilde{R}_{\delta}^{(0,1)}(u),
			\end{align*}
			with embeddings $J = \mathcal{J}_s^{\eta,\eta_{-}}$ and $J_0$ denotes the identity map. In the context of \Cref{lemma:fixedpointsmooth}, the map $g$ is given by $\mathcal{G}^\eta$ due to the linearity of the embedding $J$. Because $(s,y_0) \mapsto U(\cdot,s)y_0, s \mapsto \mathcal{K}_s^\eta$ and $u \mapsto \Tilde{R}_\delta(u)$ are $C^1$-smooth (\Cref{sec:Floquet}, \Cref{prop:ketas} and \Cref{lemma:smoothness3}), the map $g$ is $C^1$-smooth and one can easily verify the additional equalities. The second condition follows from \eqref{eq:contractionregularity} and the fact that the Lipschitz constant is independent of $s \in \mathbb{R}$ due to \Cref{prop:ketas}. The third condition follows from the fact that $\Psi$ is given by $\hat{u}^{\eta_{-}}$ and therefore well-defined due to \Cref{thm:contraction}. The mentioned results show that the fourth condition is satisfied. It follows from \Cref{prop:ketas} and \Cref{lemma:smoothness3} that the fifth condition is satisfied as well. Hence, we conclude that the map $\mathcal{J}_s^{\eta,\eta_{-}} \circ \hat{u}^{\eta_{-}}$ is of the class $C^1$ and that $D(\mathcal{J}_s^{\eta,\eta_{-}} \circ \hat{u}^{\eta_{-}}) = \mathcal{J}_s^{\eta,\eta_{-}} \circ \hat{u}^{\eta_{-},(1)} \in \mathcal{L}(E_0,\BC_s^\eta)$, where $\hat{u}^{\eta_{-},(1)}(s,y_0)$ is the unique solution of
			\begin{align*}
				w^{(1)} &= \mathcal{K}_s^{\eta_{-}} \circ \Tilde{R}_{\delta}^{(0,1)}(\hat{u}^{\eta_{-}}(s,y_0))w^{(1)} + U_s^{\eta_{-}}
				=: F_{\eta_{-}}^{(1)}(w^{(1)},s,y_0)
			\end{align*}
			in the space $\mathcal{L}(E_0,\BC_s^{\eta_{-}})$. Here 
			$$
			F_{\eta_{-}}^{(1)}: \mathcal{L}(E_0,\BC_s^{\eta_{-}}) \times E_0 \to \mathcal{L}(E_0,\BC_s^{\eta_{-}})
			$$ 
			and notice that $F_{\eta_{-}}^{(1)} (\cdot,s,y_0)$ is a uniform contraction (\Cref{lemma:smoothness3}), which proves the uniqueness of the fixed point.
			
			To specify the induction hypothesis, consider any integer $1 \leq l < k$ and suppose that for all $1 \leq q \leq l$ and all $\eta \in (q\eta_{-},\eta_{+}]$ that the map $\mathcal{J}_s^{\eta,\eta_{-}} \circ \hat{u}^{\eta_{-}}$ is $C^q$-smooth with $D^q(\mathcal{J}_s^{\eta,\eta_{-}} \circ \hat{u}^{\eta_{-}}) = \mathcal{J}_s^{\eta,\eta_{-}} \circ \hat{u}^{\eta_{-},(q)} \in \mathcal{L}^{q}(E_0,\BC_s^{q \eta})$, where $\hat{u}^{\eta_{-},(q)}$ is the unique solution of 
			\begin{align*}
				w^{(l)} &= \mathcal{K}_s^{l \eta_{-}} \circ \Tilde{R}_{\delta}^{(0,1)}(\hat{u}^{\eta_{-}}(s,y_0))w^{(l)} + H_{\eta_{-}}^{(l)}(s,y_0)
				=: F_{l\eta_{-}}^{(l)}(w^{(l)},s,y_0)
			\end{align*}
			in the space $\mathcal{L}^q(E_0,\BC_s^{q \eta_{-}})$. Here $H_{\eta_{-}}^{(1)}(s,y_0) = U_s^{\eta_{-}}y_0$ and for $\nu \in [\eta_{-},\eta_{+}]$ and $l \geq 2$ we have that $H_\nu^{(l)}(s,y_0)$ is a finite sum of terms of the form
			\begin{align*}
				\mathcal{K}_s^{l \nu} \circ \Tilde{R}_\delta^{(0,q)}(\hat{u}^{\eta_{-}}(s,y_0))
				(\hat{u}^{\eta_{-},(r_1)}(s,y_0),\dots,\hat{u}^{\eta_{-},(r_q)}(s,y_0)),
			\end{align*}
			with $2 \leq q \leq l$ and $1 \leq r_i < l$ for $i=1,\dots,q$ such that $r_1+\dots+r_q = l$. Here $F_{l\eta}^{(l)} : \mathcal{L}^l(E_0,\BC_s^{l \eta}) \times E_0 \to \mathcal{L}^l(E_0,\BC_s^{l \eta})$ is a uniform contraction (\Cref{lemma:smoothness3}) for any $\eta \in [\eta_{-},\eta_{+}]$, which guarantees the uniqueness of the fixed point.
			
			For the induction step, fix some $\eta \in ((l+1)\eta_{-},\eta_{+}]$ and choose $\sigma,\mu > 0$ such that $\eta_{-} < \sigma < (l+1) \sigma < \mu < \eta$. We show that \Cref{lemma:fixedpointsmooth} applies with the Banach spaces $Y_0 = \mathcal{L}^l(E_0,\BC_s^{l\sigma}), Y = \mathcal{L}^{l}(E_0,\BC_s^\mu), Y_1 = \mathcal{L}^l(E_0,\BC_s^\eta)$ and $\Lambda = E_0$, and operators
			\begin{align*}
				f(u,s,y_0)&=\mathcal{K}_s^{\mu} \circ \Tilde{R}_\delta^{(0,1)}(\hat{u}^{\eta_{-}}(s,y_0))u 
				+ H_{\mu/l}^{(l)}(s,y_0), \\
				f^{(1)}(u,s,y_0) &= \mathcal{K}_s^{\mu} \circ \Tilde{R}_\delta^{(0,1)}(\hat{u}^{\eta_{-}}(s,y_0)) 
				\in \mathcal{L}(\mathcal{L}^l(E_0,\BC_s^\mu)), \\
				f_1^{(1)}(u,s,y_0) &= \mathcal{K}_s^{\eta} \circ \Tilde{R}_\delta^{(0,1)}(\hat{u}^{\eta_{-}}(s,y_0)) 
				\in \mathcal{L}(\mathcal{L}^l(E_0,\BC_s^\eta)).
			\end{align*}
			To verify the first condition, we have to check that $g : \mathcal{L}^l(E_0,\BC_s^{l \sigma}) \times E_0 \to \mathcal{L}(E_0,\BC_s^\eta)$ given by
			\begin{align*}
				g(u,s,y_0) = \mathcal{K}_s^{\eta} \circ \Tilde{R}_\delta^{(0,1)}(\hat{u}^{\eta_{-}}(s,y_0))u 
				+ \mathcal{J}_s^{\eta,\mu} \circ H_{\mu/l}^{(l)}(s,y_0)
			\end{align*}
			is $C^1$-smooth, where now $\mathcal{J}_s^{\eta,\mu} : \mathcal{L}^l(E_0,\BC_s^\mu) \hookrightarrow \mathcal{L}^l(E_0,\BC_s^\eta)$ is the continuous embedding. Clearly, $g$ is $C^1$-smooth in the first variable since it is linear. For the second variable, notice that the map $(s,y_0) \mapsto \mathcal{K}_s^{\eta} \circ \Tilde{R}_\delta^{(0,1)}(\hat{u}^{\eta_{-}}(s,y_0))u$ is $C^1$-smooth due to \Cref{lemma: smoothness4} with $\mu > (l+1)\sigma$ and the $C^1$-smoothness of $(s,y_0) \mapsto \mathcal{J}_s^{\sigma,\eta_{-}} \hat{u}^{\eta_{-}}(s,y_0)$ for any $\sigma \geq \eta_{-}$. For the $C^1$-smoothness of the map $H_{\mu/l}^{(l)}$, we get differentiability from \Cref{lemma: smoothness4} and so we have that the derivative of this map is a finite sum of terms of the form
			\begin{align*}
				&\mathcal{K}_s^\mu \circ \tilde{R}_{\delta}^{(0,q+1)}(\hat{u}^{\eta_{-}}(s,y_0))(\hat{u}^{\eta_{-},(r_1)}(s,y_0),\dots, \hat{u}^{\eta_{-},(r_q)}(s,y_0)) \\
				&+ \sum_{j=1}^q \mathcal{K}_s^\mu \circ \tilde{R}_{\delta}^{(0,q)}(\hat{u}^{\eta_{-}}(s,y_0))(\hat{u}^{\eta_{-},(r_1)}(s,y_0),\dots,\hat{u}^{\eta_{-},(r_j + 1)}(s,y_0),\dots, \hat{u}^{\eta_{-},(r_q)}(s,y_0))
			\end{align*}
			and each $\hat{u}^{\eta_{-},(r_j)}(s,y_0)$ is a map from $E_0$ into $\BC_s^{j \sigma}$ for $j=1,\dots,q$. An application of \Cref{lemma:smoothness3} with $\mu > (l+1) \sigma$ ensures the continuity of $DH_{\mu/l}^{(l)}(s,y_0)$ and consequently that of $\mathcal{J}_s^{\eta,\mu} DH_{\mu/l}^{(l)}(s,y_0)$. The remaining calculations from the first condition are then easily checked, and condition four can be proven similarly. The Lipschitz condition and boundedness for the second condition follows by the choice of $\delta > 0$ chosen at the beginning of the proof and the contractivity of $H_{\mu/l}^{(l)}$ described above. To prove the third condition, observe that one can write
			\begin{equation*}
				\mathcal{K}_s^\eta \circ \Tilde{R}_{\delta}^{(0,1)}(\hat{u}^{\eta_{-}}(s,y_0)) = \mathcal{J}_s^{\eta,\mu} \mathcal{K}_s^\mu \circ \Tilde{R}_{\delta}^{(0,1)}(\hat{u}^{\eta_{-}}(s,y_0))
			\end{equation*}
			and applying \Cref{lemma:smoothness3} together with the $C^1$-smoothness of $\hat{u}^{\eta_{-}}$ to obtain the continuity of $(s,y_0) \mapsto \Tilde{R}_{\delta}^{(0,1)}(\hat{u}^{\eta_{-}}(s,y_0))$. This also proves the fifth condition, and so we conclude that $\hat{u}^{\eta_{-}} : E_0 \mapsto \mathcal{L}^l(E_0,\BC_s^\eta)$ is of the class $C^1$ with derivative $\hat{u}^{\eta_{-},(l+1)} = D\hat{u}^{\eta_{-},(l)} \in \mathcal{L}^{l+1}(E_0,\BC_s^\eta)$ that is the unique solution of
			\begin{equation*}
				w^{(l+1)} = \mathcal{K}_s^\mu \circ \Tilde{R}_{\delta}^{(0,1)}(\hat{u}^{\eta_{-},(l+1)})w^{(l+1)} + H_{\mu / (l+1)}^{(l+1)}(s,y_0),
			\end{equation*}
			where 
			\begin{align*}
				H_{\mu / (l+1)}^{(l+1)}(s,y_0) = \mathcal{K}_s^\mu \circ \Tilde{R}_{\delta}^{(0,2)}(\hat{u}^{\eta_{-}}(s,y_0)) (\hat{u}^{\eta_{-},(l)}(s,y_0),\hat{u}^{\eta_{-},(1)}(s,y_0)) + DH_{\mu/l}^{(l)}(s,y_0).
			\end{align*}
			A similar argument as in the proof of the $l=k=1$ case shows that the unique fixed point $\hat{u}^{\eta_{-},(l+1)}$ is also contained in $\mathcal{L}^{l+1}(E_0,\BC_s^{(l+1)\eta_{-}})$. Hence, the map $\mathcal{J}_s^{\eta,\eta_{-}} \circ \hat{u}^{\eta_{-}}$ is of the class $C^{l+1}$ provided that $\eta \in ((l+1)\eta_{-},\eta_{+}]$ and $\delta > 0$ is sufficiently small.
		\end{proof}
		\begin{theorem} \label{thm:smoothnessC}
			The map $\mathcal{C} : E_0 \to \mathbb{R}^n$ from \eqref{eq:mapC} is $C^k$-smooth.
		\end{theorem}
		\begin{proof}
			Let $ \eta \in [\eta_{-},\eta_{+}] \subset (0,\min\{-a,b\})$ such that $k \eta_{-} < \eta_{+}$. Let $\eva_s$ denote the bounded linear evolution operator (at time $s$) defined in the proof of \Cref{prop:bundle}. Recall that $\mathcal{C}(s,y_0) = \hat{u}^\eta(s,y_0)(s) = \eva_s(\hat{u}^\eta(s,y_0)),$ and so $\mathcal{C}(s,y_0) = \eva_s(\mathcal{J}_s^{\eta,\eta_{-}} \hat{u}^{\eta_{-}}(s,y_0))$. The result follows now from \Cref{thm:smoothnesscmt}. 
		\end{proof}
		
		To study in \Cref{prop:bundle} the tangent bundle of the center manifold, we have to use the partial derivative of the map $\mathcal{C}$ in the second component. The following result shows that such (higher order) partial derivatives are uniformly Lipschitz continuous.
		\begin{corollary} \label{cor:LipschitzC}
			For each $l \in \{0,\dots,k\}$, there exists a constant $L(l) > 0$ such that 
			$$\|D_2^l \mathcal{C}(s,y_0) - D_2^l \mathcal{C}(s,z_0) \| \leq L(l) \|y_0 - z_0\|
			$$ for all $(s,y_0),(s,z_0) \in E_0$.
		\end{corollary}
		\begin{proof}
			For $l = 0$, the result is already proven in \Cref{lemma:lipschitzCMT}. Now let $l \in \{1,\dots,k\}$. Then, from the proof of \Cref{thm:smoothnesscmt} we see that $\hat{u}^{\eta_{-},(l)}$ is the unique solution of a fixed point problem, where the right hand-side is a contraction with a Lipschitz constant $L(l)$ independent of $s$. Using the same strategy as the proof of \Cref{lemma:lipschitzCMT}, we obtain the desired result. 
		\end{proof}
		
		\bibliographystyle{siamplain}
		\bibliography{references}
		
	\end{sloppypar}
\end{document}